\documentclass[10pt,a4paper]{amsart}

\usepackage{amsthm,amsmath,amssymb,amsxtra,amscd,enumerate}
\usepackage[all]{xy}
\usepackage{hyperref}
\numberwithin{equation}{section}
\allowdisplaybreaks[3]

\theoremstyle{plain}
\newtheorem{thm}{Theorem}[section]
\newtheorem*{mainthm}{Main Theorem}
\newtheorem{prop}[thm]{Proposition}
\newtheorem{lem}[thm]{Lemma}

\theoremstyle{definition}
\newtheorem{de}[thm]{Definition}
\newtheorem{rem}[thm]{Remark}
\newtheorem{set}[thm]{Setting}
\newtheorem{ex}[thm]{Example}

\newcommand{\bb}{\mathbb}
\newcommand{\cal}{\mathcal}
\newcommand{\ovl}{\overline}
\newcommand{\wtl}{\widetilde}

\DeclareMathOperator*{\Hom}{Hom}
\DeclareMathOperator{\Ad}{Ad}
\DeclareMathOperator{\ad}{ad}
\DeclareMathOperator{\rank}{rank}

\DeclareMathOperator{\Spec}{Spec}

\begin{document}

%\TitleHead{Localization of Cohomological Induction}

\title{Localization of Cohomological Induction}

%\AuthorHead{Y. Oshima}

\author{Yoshiki Oshima}
\address{Kavli IPMU (WPI), The University of Tokyo, 5-1-5 Kashiwanoha, Kashiwa, 277-8583 Chiba, Japan}
\email{yoshiki.oshima@ipmu.jp}



%\thanks{Supported by Grant-in-Aid for JSPS Fellows (10J00710)}

\maketitle

\begin{abstract}
We give a geometric realization of cohomologically induced
 $(\mathfrak{g},K)$-modules.
Let $(\mathfrak{h}, L)$ be a subpair of $(\mathfrak{g},K)$.
The cohomological induction is an algebraic construction of
 $(\mathfrak{g},K)$-modules from a $(\mathfrak{h},L)$-module $V$.
For a real semisimple Lie group, the duality theorem by 
 Hecht, Mili{\v{c}}i{\'c}, Schmid, and Wolf relates
 $(\mathfrak{g},K)$-modules cohomologically induced from a Borel subalgebra
 with ${\mathcal D}$-modules on the flag variety of $\frak{g}$.
In this article we extend the theorem for more general pairs
 $(\mathfrak{g},K)$ and $(\mathfrak{h},L)$.
We consider the tensor product of a ${\mathcal D}$-module and a certain module
 associated with $V$, and prove that
 its sheaf cohomology groups are isomorphic
 to cohomologically induced modules.
\end{abstract}

\section{Introduction}
The aim of this article is to realize cohomologically induced modules
 as sheaf cohomology groups of certain sheaves on homogeneous spaces.

Cohomological induction is defined as a functor between the categories
 of $(\frak{g},K)$-modules.
Let $(\frak{g},K)$ be a pair (Definition~\ref{de:pair}) and
 let ${\cal C}(\frak{g}, K)$ be the category of $(\frak{g},K)$-modules.
Suppose that $(\frak{h},L)$ is a subpair of $(\frak{g},K)$ and that
 $K$ and $L$ are reductive.
Following the book by Knapp and Vogan \cite{KnVo}, we define the functors
 $P_{\frak{h},L}^{\frak{g},K}$ and $I_{\frak{h},L}^{\frak{g},K}
 :{\cal C}(\frak{h},L)\to {\cal C}(\frak{g},K)$
 as $V\mapsto R(\frak{g},K)\otimes_{R(\frak{h},L)} V$ and 
 $V\mapsto (\Hom_{R(\frak{h},L)}(R(\frak{g},K),V))_K$, respectively.
See Section~\ref{sec:coh} for the definition of the
 Hecke algebra $R(\frak{g},K)$.
When $\frak{g}=\frak{h}$, the functor
 $I_{\frak{h},L}^{\frak{g},K}=I_{\frak{g},L}^{\frak{g},K}$
 is called the Zuckerman functor.
Let $V$ be a $(\frak{h},L)$-module.
We define the cohomologically induced module as the $(\frak{g},K)$-module
 $(P_{\frak{h},L}^{\frak{g},K})_j(V)$ for $j\in \bb{N}$, where
 $(P_{\frak{h},L}^{\frak{g},K})_j$ is the $j$-th left derived functor
 of $P_{\frak{h},L}^{\frak{g},K}$.
Similarly, we define $(I_{\frak{h},L}^{\frak{g},K})^j(V)$,
 where $(I_{\frak{h},L}^{\frak{g},K})^j$ is the $j$-th right derived functor
 of $I_{\frak{h},L}^{\frak{g},K}$.

This construction produces a large family of representations of
 real reductive Lie groups.
Let $G_\bb{R}$ be a real reductive Lie group
 with a Cartan involution $\theta$ so that
 the group of fixed points $K_\bb{R}:=(G_\bb{R})^\theta$
 is a maximal compact subgroup.
Let $\frak{g}$ be the complexified Lie algebra of $G_\bb{R}$
 and $K$ the complexification of $K_\bb{R}$.
We give examples of cohomologically induced $(\frak{g},K)$-modules below.
In the following three examples we
 suppose that $\frak{h}$ is a parabolic subalgebra of $\frak{g}$
 and $L$ is a maximal reductive subgroup
 of the normalizer $N_{K}(\frak{h})$.
We also suppose that $V$ is a one-dimensional $(\frak{h},L)$-module.
\begin{itemize}
\item
We assume the rank condition
 $\rank \frak{g}= \rank K$
 and that $\frak{h}$ is a $\theta$-stable Borel subalgebra.
Then under a certain positivity condition on $V$,
 $(P_{\frak{h},L}^{\frak{g},K})_s(V)$
 (or $(I_{\frak{h},L}^{\frak{g},K})^s(V)$)
 is the underlying $(\frak{g},K)$-module
 of a discrete series representation of $G_\bb{R}$.
Here $s=\frac{1}{2}\dim K/L$.
\item
Suppose that $\frak{h}$ is a $\theta$-stable parabolic subalgebra.
Then the $(\frak{g},K)$-module
 $(P_{\frak{h},L}^{\frak{g},K})_s(V)$
 (or $(I_{\frak{h},L}^{\frak{g},K})^s(V)$)
 is called Zuckerman's derived functor module $A_\frak{h}(\lambda)$.
Here $s=\frac{1}{2}\dim K/L$.
\item
Let $P_\bb{R}$ be a parabolic subgroup of $G_\bb{R}$
 and suppose that $\frak{h}$ is its complexified Lie algebra.
Then  $(P_{\frak{h},L}^{\frak{g},K})_0(V)$
 (or $(I_{\frak{h},L}^{\frak{g},K})^0(V)$)
 is the underlying $(\frak{g},K)$-module
 of a degenerate principal series representation
 realized on the real flag variety $G_\bb{R}/P_\bb{R}$.
\end{itemize}

The localization theory 
 by Be{\u\i}linson--Bernstein~\cite{BB81}
 provides another important
 construction of $(\frak{g},K)$-modules.
It gives a realization of $(\frak{g},K)$-modules as
 $K$-equivariant twisted ${\cal D}$-modules
 on the full flag variety $X$ of $\frak{g}$.

These two constructions are related by a result of
 Hecht--Mili$\rm{\check{c}}$i$\rm{\acute{c}}$--Schmid--Wolf~\cite{HMSW}.
We now recall their theorem.
Let $G_\bb{R}$ be a connected real reductive Lie group and 
 let $(\frak{g},K)$ be the pair defined in the above way.
Suppose that $\frak{h}=\frak{b}$ is a Borel subalgebra of $\frak{g}$
 and $L$ is a maximal reductive subgroup of
 the normalizer $N_{K}(\frak{b})$.
Let $X$ be the full flag variety of $\frak{g}$,
 $Y$ the $K$-orbit through $\frak{b}\in X$,
 and $i:Y\to X$ the inclusion map.
Suppose that $V$ is a $(\frak{b},L)$-module and $\frak{b}$ acts as
 scalars given by $\lambda\in\frak{b}^*:=\Hom_\bb{C}(\frak{b},\bb{C})$.
Write ${\cal V}_Y$ for the corresponding locally free ${\cal O}_Y$-module
 on $Y$ and view it as a twisted ${\cal D}$-module.
Let ${\cal D}_{X,\lambda}$ be the ring of twisted differential operators
 on $X$ corresponding to $\lambda$ and 
 define the ${\cal D}_{X,\lambda}$-module
 direct image $i_+ {\cal V}_Y$.
Then the following is called the duality theorem:
\begin{thm}[\cite{HMSW}]
\label{dual}
There is an isomorphism of $(\frak{g},K)$-modules
\begin{align*}
{\rm H}^s(X, i_+{\cal V}_Y)^* \simeq
 (I_{\frak{b},L}^{\frak{g},K})^{u-s}
\Bigl(V^*\otimes\bigwedge^{\rm top}(\frak{g}/\frak{b})^*\Bigr)
\end{align*}
for $s\in \bb{N}$ and $u=\dim K/L -\dim Y$.
Here the left side is the $K$-finite dual of the $(\frak{g},K)$-module
 ${\rm H}^s(X, i_+{\cal V}_Y)$.
\end{thm}
In \cite{HMSW},
 they proved the theorem by describing the cohomology groups
 of the both sides by
 using standard resolutions and giving an isomorphism between
 the two complexes.
We note that by using the dual isomorphism (\cite[Theorem 3.1]{KnVo})
 $(P_{\frak{h},L}^{\frak{g},K})_j(V)^*
\simeq (I_{\frak{h},L}^{\frak{g},K})^j(V^*)$,
Theorem~\ref{dual} is deduced from
\begin{align}
\label{eq:HMSWisom}
{\rm H}^s(X, i_+{\cal V}_Y) \simeq
 (P_{\frak{b},L}^{\frak{g},K})_{u-s}
\Bigl(V\otimes\bigwedge^{\rm top}(\frak{g}/\frak{b})\Bigr).
\end{align} 

The relation between the cohomological induction
 and the localization has been studied further
 (see \cite{Bien}, \cite{Chang}, \cite{Kit}, \cite{MiPa}, \cite{Sch91}).
Mili{\v{c}}i{\'c}--Pand{\v{z}}i{\'c}~\cite{MiPa} gave a more conceptual proof
 of Theorem~\ref{dual} by using equivariant derived categories.
In \cite{Chang} and \cite{Kit}, Theorem~\ref{dual} was extended to the case
 of partial flag varieties.

In this article we will realize geometrically
 the cohomologically induced modules
 in the following setting.
Let $i:K\to G$ be a homomorphism between complex linear algebraic groups.
Suppose that $K$ is reductive and the kernel of $i$ is finite
 so that the pair $(\frak{g},K)$ is defined.
Let $H$ be a closed subgroup of $G$.
Put $M:=i^{-1}(H)$ and take a Levi decomposition $M=L\ltimes U$.
We write $i:Y=K/M \to G/H=X$ for the natural immersion.
Let $V$ be a $(\frak{h},M)$-module.
We see $V$ as a $(\frak{h}, L)$-module by restriction and
 define the cohomologically induced module
 $(P_{\frak{h},L}^{\frak{g},K})_j(V)$.
In this generality, we can no longer realize it as
 a (twisted) ${\cal D}$-module on $X=G/H$.
Instead we use the tensor product of an $i^{-1}{\cal D}_X$-module
 and an $i^{-1}{\cal O}_X$-module associated with $V$ which is
 equipped with a $(\frak{g},K)$-action
 (see Definition~\ref{de:assmod}).
We now state the main theorem of this article.
\begin{mainthm}[Theorem~\ref{loccoh}]
Suppose that ${\cal V}$ is an $i^{-1}\wtl{\frak{g}}_X$-module
 associated with $V$ (see Definition~\ref{de:assmod}).
Then we have an isomorphism of $(\frak{g},K)$-modules
\begin{align*}
{\rm H}^s(Y, i^{-1}i_+{\cal L}
 \otimes_{i^{-1}{\cal O}_X} {\cal V}) \simeq
 (P_{\frak{h},L}^{\frak{g},K})_{u-s}
 \Bigl(V\otimes \bigwedge^{\rm top}(\frak{g}/\frak{h})\Bigr)
\end{align*}
for $s\in\bb{N}$ and $u = \dim U$.
\end{mainthm}
Here ${\cal L}$ is the invertible sheaf on $Y$ defined in the beginning
 of Section~\ref{sec:loc} and
 the direct image $i_+{\cal L}$
 in the categories of ${\cal D}$-modules
 is defined as 
\[
i_*\bigl( ({\cal L}\otimes_{{\cal O}_Y} {\Omega_Y})
\otimes_{{\cal D}_Y}  i^*{\cal D}_X\bigr)
 \otimes_{{\cal O}_X} \Omega_X\spcheck.
\]
Hence its inverse image $i^{-1}i_+{\cal L}$
 as a sheaf of abelian groups is given by 
\[
({\cal L}\otimes_{{\cal O}_Y} {\Omega_Y})
\otimes_{{\cal D}_Y}  i^*{\cal D}_X
 \otimes_{i^{-1}{\cal O}_X} i^{-1}\Omega_X\spcheck.
\]
We note that if $V$ comes from an algebraic $H$-module, then
 we can take ${\cal V}$ to be $i^{-1}{\cal V}_X$,
 where ${\cal V}_X$ is a $G$-equivariant locally free
 ${\cal O}_X$-module with typical fiber $V$
 (Example~\ref{vectbdle}).

The work of this article was motivated by the study of branching laws
 of representations.
In \cite{Os11} a special case of Theorem~\ref{loccoh} was proved and 
 it was used to get an estimate of the restriction of $A_\frak{q}(\lambda)$
 to reductive subalgebras.

This article is organized as follows.
In Section~\ref{sec:coh} we recall the definition of cohomological induction
 following \cite{KnVo}.
In Section~\ref{sec:Ind} we give a definition of
 an $i^{-1}{\cal O}$-module associated with a $(\frak{h},M)$-module $V$.
We state and prove the main theorem (Theorem~\ref{loccoh}) in
Section~\ref{sec:loc}.
Our proof basically follows the proof of the duality theorem
 in \cite{HMSW}.
Section~\ref{sec:const} is devoted to the construction of
 an $i^{-1}{\cal O}$-module associated with a $(\frak{h},M)$-module $V$,
 which can be used for the geometric realization
 of cohomologically induced modules.
In Section~\ref{sec:tdo}, we see that the module
 $i^{-1}i_+{\cal L} \otimes {\cal V}$ can be viewed as 
 a twisted ${\cal D}$-module if $\frak{h}$ acts as scalars on $V$.
Therefore, Theorem~\ref{loccoh} becomes the isomorphism
 \eqref{eq:HMSWisom} and hence Theorem~\ref{dual} in the particular setting.

\section{Cohomological induction}\label{sec:coh}
In this section we recall the definition
 of the cohomological induction 
following \cite{KnVo}.

Let $K$ be a complex reductive algebraic group
 and let $K_\bb{R}$ be a compact real form.
Since any locally finite action of $K_\bb{R}$
 is uniquely extended to an algebraic action of $K$,
 the locally finite $K_\bb{R}$-modules are identified with
 the algebraic $K$-modules.
Define the Hecke algebra $R(K_\bb{R})$ as the space of
 $K_\bb{R}$-finite distributions on $K_\bb{R}$. 
For $S\in R(K_\bb{R})$, the pairing with a smooth function $f$
 on $K_\bb{R}$ is written as  
\[
\int_{K_\bb{R}} f(k) dS(k).
\]
The product of $S, T\in R(K_\bb{R})$ is given by 
\[
S*T: f\mapsto \int_{K_\bb{R}\times K_\bb{R}} f(kk') dS(k)dT(k').
\]
The associative algebra $R(K_\bb{R})$
 does not have the identity, but has an approximate identity
 (see \cite[Chapter I]{KnVo}).
The locally finite $K_\bb{R}$-modules are identified with the 
 approximately unital left $R(K_\bb{R})$-modules. 
The action map $R(K_\bb{R})\times V\to V$ is given by
\[
(S, v)\mapsto \int_{K_\bb{R}} kv \,dS(k)
\]
for a locally finite $K_\bb{R}$-module $V$. 
Here, $kv$ is regarded as a smooth function on $K_\bb{R}$
 that takes values on $V$.
As a $\bb{C}$-algebra, we have a natural isomorphism 
 \[R(K_\bb{R})\simeq \bigoplus_{\tau\in \widehat{K}}
 \operatorname{End}_\bb{C} (V_\tau),\] 
where $\widehat{K}$ is the set of equivalence classes of
 irreducible $K$-modules, and $V_\tau$ is a representation space
 of $\tau\in {\widehat{K}}$.
Hence $R(K_\bb{R})$ depends only on the complexification $K$
 up to natural isomorphisms, 
 so in what follows, we also denote $R(K_\bb{R})$ by $R(K)$.

\begin{de}
\label{de:pair}
Let $\frak{g}$ be a Lie algebra and $K$ a complex
 linear algebraic group such that the Lie algebra $\frak{k}$
 of $K$ is a subalgebra of $\frak{g}$.
Suppose that a homomorphism $\phi:K\to {\operatorname{Aut}}(\frak{g})$
 of algebraic groups is given, where ${\operatorname{Aut}}(\frak{g})$
 is the automorphism group of $\frak{g}$.
We say $(\frak{g},K)$ is a \emph{pair} if 
\begin{itemize}
\item
$\phi(\cdot)|_\frak{k}$ is equal to the adjoint action
 ${\Ad}_\frak{k}(K)$ of $K$, and
\item
the differential of $\phi$ is equal to the adjoint action 
${\ad}_\frak{g}(\frak{k})$.
\end{itemize}
\end{de}

Let $i: K \to G$ be a homomorphism of complex linear algebraic groups
 with finite kernel and let $\frak{g}$ be the Lie algebra of $G$. 
Then $(\frak{g},K)$ with the homomorphism $\phi:={\Ad}\circ i$
 is a pair in the above sense.

\begin{de}
Let $(\frak{g},K)$ be a pair. 
Let $V$ be a complex vector space with a Lie algebra action of $\frak{g}$ and
 an algebraic action of $K$.
We say that $V$ is a \emph{$(\frak{g},K)$-module} if 
\begin{itemize}
\item
the differential of the action of $K$ coincides with
 the restriction of the action of $\frak{g}$ to $\frak{k}$, and
\item
$(\phi(k)\xi)v=k(\xi (k^{-1}(v)))$ for $k\in K$,
 $\xi\in\frak{g}$, and $v\in V$.
\end{itemize}
\end{de}

For a pair $(\frak{g},K)$, we
 denote by ${\cal C}(\frak{g},K)$ the category of $(\frak{g},K)$-modules.
Suppose moreover that $K$ is reductive.
We extend the representation $\phi: K\to \operatorname{Aut}(\frak{g})$ to
 a representation on the universal enveloping algebra
 $\phi: K\to \operatorname{Aut}(U(\frak{g}))$.
We define the Hecke algebra $R(\frak{g},K)$ as
 \[R(\frak{g},K):= R(K)\otimes_{U(\frak{k})} U(\frak{g}).\]
The product is given by
\[(S\otimes \xi)\cdot(T\otimes \eta)=
\sum_i (S * (\langle\xi_i^*, \phi(\cdot)^{-1}\xi\rangle T)\otimes \xi_i \eta)
\]
for $S, T\in R(K)$ and $\xi,\eta\in U(\frak{g})$.
Here $\xi_i$ is a basis of the linear span of $\phi(K)\xi$ and 
 $\xi_i^*$ is its dual basis.
We regard $\langle\xi_i^*, \phi(\cdot)^{-1}\xi\rangle$
 as a function on $K_\bb{R}$.
As in the group case, 
 the $(\frak{g},K)$-modules are identified with the 
 approximately unital left $R(\frak{g}, K)$-modules. 
The action map $R(\frak{g},K)\times V\to V$ is
 given by
\[
(S\otimes \xi,\, v)\mapsto \int_{K_\bb{R}} k(\xi v)\,dS(k)
\]
for a $(\frak{g}, K)$-module $V$.

Let $(\frak{g},K)$ and $(\frak{h},L)$ be pairs in the sense of
 Definition~\ref{de:pair}.
Suppose that $K$ and $L$ are reductive.
Let $i:(\frak{h},L)\to (\frak{g},K)$ be a map between pairs, namely,
 a Lie algebra homomorphism $i_{\rm alg}:\frak{h}\to \frak{g}$ and
 an algebraic group homomorphism $i_{\rm gp}:L\to K$ satisfy the following
 two assumptions.
\begin{itemize}
\item 
The restriction of $i_{\rm alg}$ to the Lie algebra $\frak{l}$ 
 of $L$ is equal to the differential of $i_{\rm gp}$.
\item 
$\phi_K(i_{\rm gp}(l))\circ i_{\rm alg}
 =i_{\rm alg}\circ\phi_L(l)$ for $l\in L$,
 where $\phi_K$ denotes $\phi$ for $(\frak{g},K)$ in
 Definition~\ref{de:pair} and $\phi_L$ denotes $\phi$ for $(\frak{h},L)$.
\end{itemize}
We define the functors
 $P_{\frak{h},L}^{\frak{g},K}, I_{\frak{h},L}^{\frak{g},K}:
 {\cal C}(\frak{h},L)\to {\cal C}(\frak{g},K)$ by 
\begin{align*}
&P_{\frak{h},L}^{\frak{g},K}
:V\mapsto R(\frak{g},K)\otimes_{R(\frak{h},L)} V, \\
&I_{\frak{h},L}^{\frak{g},K}
:V\mapsto ({\Hom}_{R(\frak{h},L)} (R(\frak{g},K),\,V))_{K},
\end{align*}
where $(\cdot)_{K}$ is the subspace of $K$-finite vectors.
Then $P_{\frak{h},L}^{\frak{g},K}$ is right exact
 and $I_{\frak{h},L}^{\frak{g},K}$ is left exact.
Write $(P_{\frak{h},L}^{\frak{g},K})_j$ for the $j$-th
 left derived functor of $P_{\frak{h},L}^{\frak{g},K}$ and
 write $(I_{\frak{h},L}^{\frak{g},K})^j$ for the $j$-th
 right derived functor of $I_{\frak{h},L}^{\frak{g},K}$.
We can see that
 $I_{\frak{h},L}^{\frak{g},K}$ is the right adjoint functor
 of the forgetful functor
\[{\rm For}_{\frak{g},K}^{\frak{h},L}:
{\cal C}(\frak{g},K)\to {\cal C}(\frak{h},L),\quad V\mapsto
R(\frak{g},K)\otimes_{R(\frak{g},K)} V \simeq V\]
and $P_{\frak{h},L}^{\frak{g},K}$ is the left adjoint functor of
 the functor
\[{\rm For}\spcheck{}_{\frak{g},K}^{\frak{h},L}:
{\cal C}(\frak{g},K)\to {\cal C}(\frak{h},L),\quad V\mapsto
 ({\Hom}_{R(\frak{g},K)}(R(\frak{g},K), V))_{L}.\]
For a $(\frak{h},L)$-module $V$, the $(\frak{g},K)$-modules
  $(P_{\frak{h},L}^{\frak{g},K})_j(V)$ 
 and $(I_{\frak{h},L}^{\frak{g},K})^j(V)$
 are called cohomologically induced modules.

%%%%%%%%%%%%%%%%%%%%%%%%%%%%%%%%%%%%%%%%%%%%%%%%%
%%%%%%%%%%%%%%%%%%%%%%%%%%%%%%%%%%%%%%%%%%%%%%%%%
%%%%%%%%%%%%%%%%%%%%%%%%%%%%%%%%%%%%%%%%%%%%%%%%%
%%%%%%%%%%%%%%%%%%%%%%%%%%%%%%%%%%%%%%%%%%%%%%%%%
%%%%%%%%%%%%%%%%%%%%%%%%%%%%%%%%%%%%%%%%%%%%%%%%%
%%%%%%%%%%%%%%%%%%%%%%%%%%%%%%%%%%%%%%%%%%%%%%%%%
%%%%%%%%%%%%%%%%%%%%%%%%%%%%%%%%%%%%%%%%%%%%%%%%%

\section{${\cal O}$-modules associated with $(\frak{g},K)$-modules}
\label{sec:Ind}

Let $G$ be a complex linear algebraic group
 acting on a variety (or more generally a scheme) $X$.
Let $a:G\times X \to X$ be the action map and $p_2: G\times X \to X$
 the second projection.
Write ${\cal O}_X$ for the structure sheaf of $X$ and
 $a^*$, $p_2^*$ for the inverse image functors as ${\cal O}$-modules.
We say that an ${\cal O}_X$-module ${\cal M}$ is \emph{$G$-equivariant} if
 there is an isomorphism $a^*{\cal M}\simeq p_2^*{\cal M}$
 satisfying the cocycle condition.
For a $G$-equivariant ${\cal O}_X$-module ${\cal M}$,
 the $G$-action on ${\cal M}$
 differentiates to a $\frak{g}$-action on ${\cal M}$.

\begin{de}
\label{eqvb}
Suppose that $H$ is a closed algebraic subgroup
 of $G$ and $X=G/H$ is the quotient variety.
For an algebraic $H$-module $V$, define ${\cal V}_X$ as 
 the $G$-equivariant quasi-coherent ${\cal O}_X$-module 
 that has typical fiber $V$.
\end{de}

The category of $G$-equivariant quasi-coherent ${\cal O}_X$-modules
 is equivalent to the category of algebraic $H$-modules, and
 ${\cal V}_X$ is the ${\cal O}_X$-module which
 corresponds to $V$ via this equivalence.
It also corresponds to the associated bundle $G\times_H V \to G/H$.
The local sections of ${\cal V}_X$ can be identified with
 the $V$-valued regular functions $f$ on open subsets of $G$
 satisfying $f(gh)=h^{-1}\cdot f(g)$ for $h\in H$.
We often use this identification in the following.

Note that ${\cal V}_X$ is locally free if $V$ is finite-dimensional.
Indeed, let $v_1, \dots, v_n$ be a basis of $V$ and
 take local sections $\wtl{v}_1, \dots, \wtl{v}_n$
 such that $\wtl{v}_i(e)=v_i$ for the identity element $e\in G$.
Then the map ${\cal O}_X^{\oplus n}\to {\cal V}_X$ given by
 $(f_i)_{i} \mapsto \sum_{i=1}^n f_i\wtl{v}_i$ is defined
 near the base point $eH\in G/H$ and is an isomorphism on
 some open neighborhood of $eH$.

Suppose that $X$ is a smooth $G$-variety. 
Then the infinitesimal action is defined as a Lie algebra homomorphism
 from the Lie algebra $\frak{g}$ of $G$ to the space of
 vector fields ${\cal T}(X)$ on $X$. 
Denote the image of $\xi\in\frak{g}$ by $\xi_X\in{\cal T}(X)$. 
Then $\xi_X$ gives a first-order differential operator
 on the structure sheaf ${\cal O}_X$. 
Let $\wtl{\frak{g}}_X:={\cal O}_X \otimes_\bb{C} \frak{g}$. 
This module becomes a Lie algebroid in a natural way (see \cite[\S 1.2]{BB93}):
the Lie bracket is defined by 
\[ [f\otimes \xi, g\otimes \eta]
=fg\otimes[\xi,\eta]+f \xi_X(g)\otimes \eta-g\eta_X(f)\otimes \xi
\]
for $f,g \in {\cal O}_X$ and $\xi,\eta\in\frak{g}$.
Here $f\in{\cal O}_X$ means that $f$ is a local section of ${\cal O}_X$.
Similar notation will be used for other sheaves.
Write $U(\wtl{\frak{g}}_X)(\simeq {\cal O}_X\otimes U(\frak{g}))$
 for the universal enveloping algebra of 
$\wtl{\frak{g}}_X$.
Then a $U(\wtl{\frak{g}}_X)$-module 
 is identified with 
 an ${\cal O}_X$-module ${\cal M}$ with a $\frak{g}$-action
 satisfying  $\xi(fm)=\xi_X(f)m+f(\xi m)$ for
 $\xi\in\frak{g}$, $f\in {\cal O}_X$, and $m\in {\cal M}$.

Let ${\cal T}_X$ be the tangent sheaf of $X$ and 
let 
$p: \wtl{\frak{g}}_X(={\cal O}_X\otimes_\bb{C} \frak{g})\to {\cal T}_X$
 be the map
 given by $f\otimes \xi\mapsto f\xi_X$.
Then the kernel ${\cal H}:={\ker}\, p$ is
 isomorphic to the $G$-equivariant locally free 
 ${\cal O}_X$-module with typical fiber $\frak{h}$.
Let ${\cal D}_X$ be the ring of differential operators on $X$.
The map $p$ extends to
 $p:U(\wtl{\frak{g}}_X)\to {\cal D}_X$ 
 and descends to an isomorphism of algebras
\begin{align}
\label{ugdiso}
U(\wtl{\frak{g}}_X)/U(\wtl{\frak{g}}_X){\cal H}
 \xrightarrow{\sim} {\cal D}_X.
\end{align}

We will work in the following setting.

\begin{set}
\label{setting}
{\rm
Let $i:K\to G$ be a homomorphism of complex linear algebraic groups
 with finite kernel.
Let $H$ be a closed algebraic subgroup of $G$.
Put $M:=i^{-1}(H)$, which is an algebraic subgroup of $K$, 
 and write $X:=G/H$ and $Y:=K/M$
 for the quotient varieties. 
The map $i:K\to G$ induces an injective
 morphism between the quotient varieties $i:Y\to X$ 
 and an injective homomorphism between Lie algebras
 $di:\frak{k}\to \frak{g}$.
We identify $\frak{k}$ with its image $di(\frak{k})$
 and regard $\frak{k}$ as a subalgebra of $\frak{g}$.
}
\end{set}

In particular, $(\frak{g},K)$ and $(\frak{h},M)$ become pairs in
 the sense of Definition~\ref{de:pair}, where $\frak{h}$ is
 the Lie algebra of $H$.

Let $e\in K$ be the identity element and
 let $o:=eM \in Y$ be the base point of $Y$.
Write
\begin{align*}
{\cal I}_Y&:=\{f\in {\cal O}_X : f(y)=0 {\rm \ for\ } y\in Y\},\\ 
{\cal I}_o&:=\{f\in {\cal O}_X : f(o)=0\},
\end{align*}
 so ${\cal I}_Y$ is the defining ideal of the closure $\ovl{Y}$ of $Y$.
It follows that $i^{-1}{\cal O}_X/{\cal I}_Y \simeq {\cal O}_Y$.
Here $i^{-1}$ denotes the inverse image functor for the sheaves of
 abelian groups.
For an $i^{-1}{\cal O}_X$-module ${\cal M}$,
 the support of the sheaf ${\cal M}/(i^{-1}{\cal I}_o){\cal M}$
 is contained in $\{o\}$ so it is regarded as a vector space.

Let $Y_p$  be the scheme $(Y, i^{-1}{\cal O}_X/({\cal I}_Y)^p)$ for $p\geq 1$.
If locally we have $X=\Spec A$, $Y=\Spec I$, and $Y$ is closed in $X$,
 then $Y_p=\Spec (A/I^p)$.
The scheme $Y_1$ is identified with the algebraic variety $Y$.
If ${\cal M}$ is an $i^{-1}{\cal O}_X$-module,
 then the sheaf ${\cal M}/(i^{-1}{\cal I}_Y)^p{\cal M}$
 can be viewed as an ${\cal O}_{Y_p}$-module.

The inverse image $i^{-1}U(\wtl{\frak{g}}_X)$ of $U(\wtl{\frak{g}}_X)$
 is a sheaf of algebras on $Y$ and an $i^{-1}{\cal O}_X$-bimodule.
We will call $i^{-1}U(\wtl{\frak{g}}_X)$-modules simply
 $i^{-1}\wtl{\frak{g}}_X$-modules.
The $K$-action on $i^{-1}\wtl{\frak{g}}_X$ is given by
 $f\otimes \xi \mapsto (k\cdot f) \otimes {\rm Ad}(i(k))(\xi)$
 for $f\in i^{-1}{\cal O}_X$, $\xi\in \frak{g}$, $k \in K$.
Suppose that ${\cal M}$ is an $i^{-1}\wtl{\frak{g}}_X$-module 
and let $i^{-1}\wtl{\frak{g}}_X\otimes {\cal M}\to {\cal M}$ be the action map.
Then the inclusion $\frak{g}\cdot ({\cal I}_Y)^p\subset ({\cal I}_Y)^{p-1}$
 induces a map 
 $i^{-1}\wtl{\frak{g}}_X\otimes {\cal M}/(i^{-1}{\cal I}_Y)^p{\cal M}
 \to {\cal M}/(i^{-1}{\cal I}_Y)^{p-1}{\cal M}$. 
The $K$-actions on $X$ and $Y$ induce a $K$-action on $Y_p$.
Since $Y$ is $K$-stable in $X$, we have
 $\frak{k}\cdot ({\cal I}_Y)^p\subset ({\cal I}_Y)^p$.
Therefore, we can define a $\frak{k}$-action on 
 ${\cal M}/(i^{-1}{\cal I}_Y)^p{\cal M}$.
Similarly, we have $\frak{h}\cdot {\cal I}_o\subset {\cal I}_o$ and 
 we can equip ${\cal M}/(i^{-1}{\cal I}_o){\cal M}$
 with a $\frak{h}$-module structure.

\begin{de}
\label{de:assmod}
Let $V$ be a $(\frak{h},M)$-module.
We say an $i^{-1}\wtl{\frak{g}}_X$-module
 ${\cal V}$ is \emph{associated with $V$} if
 ${\cal V}/(i^{-1}{\cal I}_Y)^p{\cal V}$
 is a $K$-equivariant quasi-coherent ${\cal O}_{Y_p}$-module for all $p\geq 1$
 and the following five assumptions hold.
\begin{enumerate}
\item
The canonical map
\[{\cal V}/(i^{-1}{\cal I}_Y)^p{\cal V}
\to {\cal V}/(i^{-1}{\cal I}_Y)^{p-1}{\cal V}\]
 commutes with $K$-actions for $p\geq 2$.
\item
${\cal V}/(i^{-1}{\cal I}_Y)^p{\cal V}$
 is a flat ${\cal O}_{Y_p}$-module for $p\geq 1$.
\item
The action map
$i^{-1}\wtl{\frak{g}}_X\otimes {\cal V}/(i^{-1}{\cal I}_Y)^p{\cal V}
\to {\cal V}/(i^{-1}{\cal I}_Y)^{p-1}{\cal V}$
commutes with $K$-actions for $p\geq 2$.
Here $K$ acts on $i^{-1}\wtl{\frak{g}}_X\otimes {\cal V}/(i^{-1}{\cal I}_Y)^p{\cal V}$ by diagonal.
\item
The $\frak{k}$-action on
 ${\cal V}/(i^{-1}{\cal I}_Y)^p{\cal V}$
 induced from the $\frak{g}$-action on ${\cal V}$
 coincides with the differential of the $K$-action on
 ${\cal V}/(i^{-1}{\cal I}_Y)^p{\cal V}$ for $p\geq 1$.
\item
There is an isomorphism 
$\iota:
{\cal V}/(i^{-1}{\cal I}_o){\cal V}
 \xrightarrow{\sim} V$
which commutes with $\frak{h}$-actions and $M$-actions.
\end{enumerate}
\end{de}

\begin{rem}
The $\frak{g}$-action and the $K$-action on ${\cal V}$ induce
 a $\frak{h}$-action and an $M$-action on
 ${\cal V}/(i^{-1}{\cal I}_o){\cal V}$.
The conditions (3) and (4) imply
 that ${\cal V}/(i^{-1}{\cal I}_o){\cal V}$ becomes
 a $(\frak{h},M)$-module.
\end{rem}

\begin{ex}
\label{vectbdle}
Suppose that $V$ is an $H$-module and define
 the $G$-equivariant quasi-coherent ${\cal O}_X$-module ${\cal V}_X$ as
 in Definition \ref{eqvb}.
The $G$-action on ${\cal V}_X$ induces a $\frak{g}$-action
 and a $K$-action on ${\cal V}_X$.
Then by regarding $V$ as a $(\frak{h},M)$-module, 
 $i^{-1}{\cal V}_X$ is associated with $V$.
\end{ex}

We will construct an $i^{-1}\wtl{\frak{g}}_X$-module associated with
 an arbitrary $(\frak{h},M)$-module in Section \ref{sec:const}.

\begin{ex}
\label{tensor}
Let ${\cal V}$ and ${\cal W}$ be
 $i^{-1}\wtl{\frak{g}}_X$-modules associated with
 $(\frak{h},M)$-modules $V$ and $W$, respectively.
Then the tensor product
 ${\cal V}\otimes_{i^{-1}{\cal O}_X} {\cal W}$ is associated with
 the $(\frak{h},M)$-module $V\otimes W$.
\end{ex}

We can define the pull-back of $i^{-1}\wtl{\frak{g}}_X$-modules
 associated with
 $V$ in the following way.
Let $K'$, $G'$, $H'$ be another triple of
 algebraic groups satisfying the assumptions in
 Setting \ref{setting}.
In particular, the map $i':K'\to G'$ induces a morphism of
 the quotient varieties $i':K'/M'\to G'/H'$,
 where $M':=(i')^{-1}(H')$.
Suppose that $\varphi_K: K'\to K$ and $\varphi:G'\to G$ are homomorphisms
 such that the diagram
\begin{align*}
\xymatrix{
K' \ar[r]^{i'} \ar[d]_{\varphi_K}& G' \ar[d]^{\varphi}\\
K \ar[r]^{i} & G
}
\end{align*}
commutes and that 
$\varphi(H')\subset H$.
Then $\varphi_K(M')\subset M$.
The maps $\varphi$, $\varphi_K$ induce
 morphisms $\varphi : X':=G'/H' \to X$, $\varphi_K: Y':=K'/M'\to Y$
 and $\varphi_p: Y'_p:=(Y', (i')^{-1}{\cal O}_{X'}/({\cal I}_{Y'})^p)
 \to Y_p$.
We get the commutative diagram:
\begin{align*}
\xymatrix{
Y' \ar[r]^{i'} \ar[d]_{\varphi_K}& X' \ar[d]^{\varphi}\\
Y \ar[r]^{i} & X.
}
\end{align*}
Suppose that ${\cal V}$ is an $i^{-1}\wtl{\frak{g}}_X$-module associated with
 a $(\frak{h},M)$-module $V$.
Let
 ${\cal V}':=
 (i')^{-1}{\cal O}_{X'}\otimes_{(\varphi\circ i')^{-1}{\cal O}_X}
 \varphi_K^{-1}{\cal V}$.
We define a $\frak{g}'$-action on ${\cal V}'$ by
 $\xi(f\otimes v)= \xi_{X'} (f) \otimes v + f \otimes \varphi(\xi) v$
 for $\xi\in \frak{g}'$, $f\in (i')^{-1}{\cal O}_{X'}$, and
 $v\in \varphi_K^{-1}{\cal V}$ so that ${\cal V}'$ becomes
 an $(i')^{-1}\wtl{\frak{g}'}_{X'}$-module.
Since 
\[{\cal V}'/((i')^{-1}{\cal I}_{Y'})^p{\cal V}'\simeq
  (i')^{-1}{\cal O}_{X'}/({\cal I}_{Y'})^p
 \otimes_{(\varphi\circ i')^{-1}{\cal O}_X}
 \varphi_K^{-1}{\cal V}
 \simeq \varphi_p^*({\cal V}/(i^{-1}{\cal I}_Y)^p{\cal V}),
\]
 the sheaf ${\cal V}'/((i')^{-1}{\cal I}_{Y'})^p{\cal V}'$
 is a $K'$-equivariant quasi-coherent ${\cal O}_{Y'_p}$-module.
We can easily show the following proposition
 by checking the
 five assumptions in Definition~\ref{de:assmod}.

\begin{prop}
\label{pullback}
Let $V$ be a $(\frak{h},M)$-module and ${\cal V}$ an
 $i^{-1}\wtl{\frak{g}}_X$-module associated with $V$.
Then the $(i')^{-1}\wtl{\frak{g}'}_{X'}$-module
 $(i')^{-1}{\cal O}_{X'}\otimes_{(\varphi\circ i')^{-1}{\cal O}_X}
 \varphi_K^{-1}{\cal V}$
 is associated with
 the $(\frak{h}',M')$-module
 ${\rm For}_{\frak{h},M}^{\frak{h}', M'}(V)$.
\end{prop}

\section{Localization of cohomological induction}
\label{sec:loc}

We retain Setting \ref{setting}.
In this section, we assume moreover that $K$ is reductive.
Let $M=L\ltimes U$ be a Levi decomposition of $M$,
 where $L$ is a maximal reductive subgroup of $M$ and $U$
 is the unipotent radical of $M$.
The corresponding decomposition of the Lie algebra is
 $\frak{m}=\frak{l}\oplus\frak{u}$.

Let $V$ be a $(\frak{h}, M)$-module. 
We can see $V$ as a $(\frak{h}, L)$-module by restriction
 and then define the cohomologically induced module
 $(P_{\frak{h},L}^{\frak{g},K})_j (V)$ as in Section~\ref{sec:coh}.

In order to state the main theorem,
 we need a shift of modules by a character
 (or an invertible sheaf) that we will define in the following.
Write $\bigwedge^{\rm top}(\frak{k}/\frak{l})$ for 
 the top exterior product of $\frak{k}/\frak{l}$
 and view it as
 a one-dimensional $L$-module by the adjoint action.
Since $K$ and $L$ are reductive, the identity component of $L$
 acts trivially on $\bigwedge^{\rm top}(\frak{k}/\frak{l})$.
We extend the $L$-action on $\bigwedge^{\rm top}(\frak{k}/\frak{l})$
 to an $M$-action by letting $U$ act trivially.
Define ${\cal L}$ as the $K$-equivariant locally free
 ${\cal O}_Y$-module on $Y:=K/M$ whose
 typical fiber is isomorphic
 to the $M$-module $\bigwedge^{\rm top}(\frak{k}/\frak{l})$.
The $K$-action on ${\cal L}$ differentiates to a $\frak{k}$-action.
Then ${\cal L}$ becomes a $U(\wtl{\frak{k}}_Y)$-module
 and the kernel of the map
 $\wtl{\frak{k}}_Y \to {\cal T}_Y$ acts by zero
 because the identity component of $M$ acts trivially on
 $\bigwedge^{\rm top}(\frak{k}/\frak{l})$.
Therefore, ${\cal L}$ has a structure of left ${\cal D}_Y$-module via
 the isomorphism \eqref{ugdiso} for $Y$.

Let ${\cal M}$ be a left ${\cal D}_Y$-module.
Recall that the direct image of ${\cal M}$ by $i$ in the category of
 left ${\cal D}$-modules is defined as
\begin{align}
\label{eq:defDpush}
i_+{\cal M}:=
 i_*(({\cal M}\otimes_{{\cal O}_Y} \Omega_Y)\otimes_{{\cal D}_Y}
 {i^*{\cal D}_X})\otimes_{{\cal O}_X} \Omega_X\spcheck,
\end{align}
where $i_*$ is the direct image functor for sheaves
 of abelian groups, 
 $\Omega_Y$ is the canonical sheaf of $Y$, and
 $\Omega_X\spcheck$ is the dual of the canonical sheaf of $X$.
Via the map $p:U(\wtl{\frak{g}}_X)\to {\cal D}_X$, 
 we can see $i_+{\cal M}$ as a $\wtl{\frak{g}}_X$-module.
The inverse image $i^{-1}i_+{\cal M}$ as a sheaf of abelian groups is
\[i^{-1}i_+{\cal M}=
 ({\cal M}\otimes_{{\cal O}_Y} \Omega_Y)\otimes_{{\cal D}_Y}
 {i^*{\cal D}_X}\otimes_{i^{-1}{\cal O}_X} i^{-1}\Omega_X\spcheck,
\]
which has an $i^{-1}\wtl{\frak{g}}_X$-module structure.
We note that the functor $i^{-1}i_+$ is exact.

Define subsheaves of ${\cal D}_X$ by
\[F_p{\cal D}_X
:=\{D\in {\cal D}_X: D({\cal I}_Y)^{p+1}\subset {\cal I}_Y\}
\]
for $p\geq 0$.
They are ${\cal O}_X$-bi-submodules of ${\cal D}_X$ 
 and form a filtration of ${\cal D}_X$.
It induces a filtration of $i^{-1}i_+{\cal L}$:
\[
F_pi^{-1}i_+{\cal L}
 := ({\cal L}\otimes_{{\cal O}_Y} \Omega_Y)\otimes_{{\cal D}_Y}
 {i^* F_p{\cal D}_X}\otimes_{i^{-1}{\cal O}_X} i^{-1}\Omega_X\spcheck.
\]
It follows from the definition of $F_p{\cal D}_X$ that 
 $F_pi^{-1}i_+{\cal L}$ is annihilated by $(i^{-1}{\cal I}_Y)^{p+1}$
 and hence is regarded as a quasi-coherent ${\cal O}_{Y_{p+1}}$-module.

Here is the main theorem of this article:

\begin{thm}
\label{loccoh}
In Setting~\ref{setting}, we assume that $K$ is reductive.
Let $M=L\ltimes U$ be a Levi decomposition.
Suppose that $V$ is a $(\frak{h}, M)$-module and 
 that ${\cal V}$ is an $i^{-1}\wtl{\frak{g}}_X$-module
 associated with $V$ (Definition~\ref{de:assmod}).
Then we have an isomorphism of $(\frak{g},K)$-modules
\begin{align*}
{\rm H}^s(Y, i^{-1}i_+{\cal L}
 \otimes_{i^{-1}{\cal O}_X} {\cal V}) \simeq
 (P_{\frak{h},L}^{\frak{g},K})_{u-s}
 \Bigl(V\otimes \bigwedge^{\rm top}(\frak{g}/\frak{h})\Bigr)
\end{align*}
for $s\in \bb{N}$ and $u= \dim U$.
(See the remark below for the definition of the $(\frak{g},K)$-action
 on the left side.)
\end{thm}

\begin{rem}
\label{rem:gkact}
Since $i^{-1}i_+{\cal L}$ and ${\cal V}$ have
 $i^{-1}\wtl{\frak{g}}_X$-module structures,
 the tensor product
 $i^{-1}i_+{\cal L} \otimes_{i^{-1}{\cal O}_X} {\cal V}$
 becomes an $i^{-1}\wtl{\frak{g}}_X$-module.
This gives a $\frak{g}$-action on the cohomology group
 ${\rm H}^s(Y, i^{-1}i_+{\cal L}
 \otimes_{i^{-1}{\cal O}_X} {\cal V})$.
In order to define a $K$-action, 
 we use the filtration $F_pi^{-1}i_+{\cal L}$ defined above.
By definition,
 $(i^{-1}{\cal I}_Y)^{p+1}$ annihilates $F_pi^{-1}i_+{\cal L}$
 and hence
\begin{align}
\label{filiso}
F_pi^{-1}i_+{\cal L}\otimes_{i^{-1}{\cal O}_X}{\cal V}
\simeq  F_pi^{-1}i_+{\cal L}
 \otimes_{{\cal O}_{Y_{q}}}{\cal V}/(i^{-1}{\cal I}_Y)^{q}{\cal V}
\end{align}
for $p<q$.
Since ${\cal V}/(i^{-1}{\cal I}_Y)^{q}{\cal V}$ is
 a flat ${\cal O}_{Y_q}$-module by Definition \ref{de:assmod} (2), 
 the map
\[F_{p-1}i^{-1}i_+{\cal L}\otimes_{i^{-1}{\cal O}_X}{\cal V}
 \to F_{p}i^{-1}i_+{\cal L}\otimes_{i^{-1}{\cal O}_X}{\cal V}\]
is injective.
We let $K$ act on the right side of \eqref{filiso} by diagonal.
Then it gives a $K$-action on  
 ${\rm H}^s(Y, F_pi^{-1}i_+{\cal L}\otimes_{i^{-1}{\cal O}_X}{\cal V})$.
Using the isomorphisms
\begin{align*}
 {\rm H}^s(Y, i^{-1}i_+{\cal L}\otimes_{i^{-1}{\cal O}_X}{\cal V})
&\simeq 
 {\rm H}^s \bigl(Y,
 (\varinjlim_p F_pi^{-1}i_+{\cal L})\otimes_{i^{-1}{\cal O}_X}{\cal V}\bigr)\\
&\simeq 
 {\rm H}^s \bigl(Y,
 \varinjlim_p (F_pi^{-1}i_+{\cal L}\otimes_{i^{-1}{\cal O}_X}{\cal V})\bigr)\\
&\simeq 
  \varinjlim_p {\rm H}^s(Y,
 F_pi^{-1}i_+{\cal L}\otimes_{i^{-1}{\cal O}_X}{\cal V}),
\end{align*}
 we define a $K$-action on 
 ${\rm H}^s(Y, i^{-1}i_+{\cal L}\otimes_{i^{-1}{\cal O}_X}{\cal V})$.
With these actions, 
 ${\rm H}^s(Y, i^{-1}i_+{\cal L}\otimes_{i^{-1}{\cal O}_X}{\cal V})$
 becomes a $(\frak{g},K)$-module because of
 Definition \ref{de:assmod} (3) and (4).
\end{rem}

\begin{proof}
Let $\wtl{X}:=G/L$ and $\wtl{Y}:=K/L$ be the quotient varieties.
We have the commutative diagram: 
\begin{align*}
\xymatrix{
\wtl{Y} \ar[r]^{\tilde{\imath}} \ar[d]_{\pi_K} 
& \wtl{X} \ar[d]^{\pi}\\
Y   \ar[r]^{i}   & X
}
\end{align*}
where the maps are defined canonically.

The direct image functor $i_+$ defined as \eqref{eq:defDpush}
 induces the direct image functor between the bounded derived categories
 of left ${\cal D}$-modules, which we denote by 
 $i_+ : {\bf D}^{\rm b}({\cal D}_Y) \to {\bf D}^{\rm b}({\cal D}_X)$.
Similarly for $\tilde{\imath}_+$, $\pi_+$, and $(\pi_K)_+$.
We have
 $\pi_+ \circ \tilde{\imath}_+\simeq i_+\circ (\pi_K)_+$.
Since $\pi_K$ is a smooth morphism and the fiber is isomorphic to
 the affine space $\bb{C}^u$, 
 it follows that $(\pi_K)_+ {\Omega}_{\wtl{Y}}\spcheck\simeq {\cal L}[u]$
 (see \cite{HMSW}).
Here ${\cal L}[u] \in {\bf D}^{\rm b}({\cal D}_Y)$
 is the complex $(\cdots \to 0\to {\cal L} \to 0 \to \cdots)$, 
 concentrated in degree $-u$.
Therefore,
 $i_+(\pi_K)_+{\Omega}_{\wtl{Y}}\spcheck\simeq i_+{\cal L}[u]$
 in ${\bf D}^{\rm b}({\cal D}_X)$.

Since $L$ is reductive,
 the varieties $\wtl{X}$ and $\wtl{Y}$ are affine
 by Matsushima's criterion.
Hence the functor $\tilde{\imath}_+$ is exact for quasi-coherent ${\cal D}$-modules and $\pi_*$ is exact for quasi-coherent ${\cal O}$-modules.

Denote by ${\cal T}_{\wtl{X}/X}$ the sheaf of local vector fields
 on $\wtl{X}$ tangent to the fiber of $\pi$, and denote by $\Omega_{\wtl{X}/X}$
 the top exterior product of its dual ${\cal T}_{\wtl{X}/X}\spcheck$.
We note that there is a natural isomorphism
 $\Omega_{\wtl{X}/X}\simeq 
 \Omega_{\wtl{X}}\otimes_{{\cal O}_{\wtl{X}}}\pi^*\Omega_X\spcheck$.
Recall that for ${\cal M}\in {\bf D}^{\rm b}({\cal D}_{\wtl{X}})$ 
 the direct image $\pi_+{\cal M}$ is defined as
\[
\pi_+{\cal M}
=\pi_*\bigl(
 ({\cal M}\otimes_{{\cal O}_{\wtl{X}}} \Omega_{\wtl{X}})
 \otimes^{\bb{L}}_{{\cal D}_{\wtl{X}}} \pi^*{\cal D}_X \bigr)
 \otimes_{{\cal O}_X} \Omega_X\spcheck.
\]
The left ${\cal D}_{\wtl{X}}$-module
 $\pi^*{\cal D}_X$ has the resolution (see \cite[Appendix A.3.3]{HMSW}):
\begin{align}
\label{resold}
{\cal D}_{\wtl{X}} \otimes_{{\cal O}_{\wtl{X}}}
 \bigwedge^\bullet {\cal T}_{\wtl{X}/X}
 \to \pi^*{\cal D}_X,
\end{align}
 where the boundary map $\partial$ on 
${\cal D}_{\wtl{X}} \otimes_{{\cal O}_{\wtl{X}}}
 \bigwedge^\bullet {\cal T}_{\wtl{X}/X}$
 is given as
\begin{align*}
 &D \otimes \wtl{\xi_1} \wedge \cdots\wedge \wtl{\xi_d} \\
&\mapsto
\sum_{i=1}^d (-1)^{i+1} D \wtl{\xi_i}
 \otimes \wtl{\xi_1} \wedge \cdots \wedge \widehat{\wtl{\xi_i}}
 \wedge \cdots \wedge \wtl{\xi_d}\\
&+\sum_{1\leq i<j\leq d} (-1)^{i+j} D \otimes
 [\wtl{\xi_i}, \wtl{\xi_j}] \wedge \wtl{\xi_1} \wedge
 \cdots \wedge \widehat{\wtl{\xi_i}} \wedge \cdots \wedge
 \widehat{\wtl{\xi_j}} \wedge \cdots \wedge \wtl{\xi_d}.
\end{align*}
The right $\pi^{-1}{\cal D}_X$-module structure
 is not canonically defined on the complex, but the $\frak{g}$-action
 can be described as
\begin{align*}
 \xi (D \otimes \wtl{\xi_1} \wedge \cdots\wedge \wtl{\xi_d})
=-D \xi_{\wtl{X}} \otimes \wtl{\xi_1} \wedge \cdots\wedge \wtl{\xi_d}
  + D \otimes \xi (\wtl{\xi_1} \wedge \cdots\wedge \wtl{\xi_d})
\end{align*}
for $\xi\in\frak{g}$.
Here we use the $\frak{g}$-action
 on $\bigwedge{\cal T}_{\wtl{X}/X}$ induced from
 the $G$-equivariant structure.

By using the resolution \eqref{resold}, 
 the direct image
 $\pi_+ \tilde{\imath}_+ {\Omega}_{\wtl{Y}}\spcheck$
 is given as the complex
\begin{align*}
\pi_*\Bigl(\tilde{\imath}_+ {\Omega}_{\wtl{Y}}\spcheck
 \otimes_{{\cal O}_{\wtl{X}}} \Omega_{\wtl{X}}
 \otimes_{{\cal O}_{\wtl{X}}} \bigwedge^\bullet {\cal T}_{\wtl{X}/X}\Bigr)
 \otimes_{{\cal O}_X} \Omega_{X}\spcheck.
\end{align*}
As a result, we have
\begin{align*}
i_+{\cal L}[u]
\simeq\pi_*\Bigl(\tilde{\imath}_+ {\Omega}_{\wtl{Y}}\spcheck
 \otimes_{{\cal O}_{\wtl{X}}} 
 \bigwedge^\bullet {\cal T}_{\wtl{X}/X}
 \otimes_{{\cal O}_{\wtl{X}}}
 \Omega_{\wtl{X}/X}\Bigl)
\end{align*}
and hence 
\begin{align}
\label{isom1}
i^{-1}i_+{\cal L}[u]
&\simeq
i^{-1}\pi_*\Bigl(\tilde{\imath}_+ {\Omega}_{\wtl{Y}}\spcheck
 \otimes_{{\cal O}_{\wtl{X}}} 
 \bigwedge^\bullet {\cal T}_{\wtl{X}/X}
 \otimes_{{\cal O}_{\wtl{X}}}
 \Omega_{\wtl{X}/X}\Bigr)\\ \nonumber
&\simeq
i^{-1}\pi_*\tilde{\imath}_*
 \Bigl(\tilde{\imath}^{-1}\tilde{\imath}_+ {\Omega}_{\wtl{Y}}\spcheck
 \otimes_{\tilde{\imath}^{-1}{\cal O}_{\wtl{X}}} 
 \tilde{\imath}^{-1}\bigwedge^\bullet {\cal T}_{\wtl{X}/X}
 \otimes_{\tilde{\imath}^{-1}{\cal O}_{\wtl{X}}}
 \tilde{\imath}^{-1}\Omega_{\wtl{X}/X}\Bigr)\\  \nonumber
&\simeq
(\pi_K)_*\biggl(\tilde{\imath}^{-1}
 \tilde{\imath}_+ {\Omega}_{\wtl{Y}}\spcheck
 \otimes_{\tilde{\imath}^{-1}{\cal O}_{\wtl{X}}} 
 \tilde{\imath}^{-1}
 \Bigl(\bigwedge^\bullet {\cal T}_{\wtl{X}/X}
 \otimes_{{\cal O}_{\wtl{X}}}
 \Omega_{\wtl{X}/X}\Bigr)\biggr).
\end{align}

There is a natural morphism of
 complexes of $i^{-1}{\cal O}_X$-modules
\begin{align}
\label{projf}
\psi:\ 
&(\pi_K)_*\biggl(\tilde{\imath}^{-1}
 \tilde{\imath}_+ {\Omega}_{\wtl{Y}}\spcheck
 \otimes_{\tilde{\imath}^{-1}{\cal O}_{\wtl{X}}} 
 \tilde{\imath}^{-1}
 \Bigl(
 \bigwedge^\bullet {\cal T}_{\wtl{X}/X}
 \otimes_{{\cal O}_{\wtl{X}}}
 \Omega_{\wtl{X}/X}
 \Bigr)\biggr)
 \otimes_{i^{-1}{\cal O}_X} {\cal V}\\ \nonumber
\to\ 
&(\pi_K)_*\biggl(\tilde{\imath}^{-1}
 \tilde{\imath}_+ {\Omega}_{\wtl{Y}}\spcheck
 \otimes_{\tilde{\imath}^{-1}{\cal O}_{\wtl{X}}} 
 \tilde{\imath}^{-1} \Bigl(
 \bigwedge^\bullet {\cal T}_{\wtl{X}/X}
 \otimes_{{\cal O}_{\wtl{X}}}
 \Omega_{\wtl{X}/X}
 \Bigr)
 \otimes_{\pi_K^{-1}i^{-1}{\cal O}_{X}} 
 \pi_K^{-1}{\cal V} \biggr).
\end{align}
We claim that $\psi$ is an isomorphism.
Indeed, if
 $F_p\tilde{\imath}^{-1}\tilde{\imath}_+{\Omega}_{\wtl{Y}}\spcheck$
 denotes the filtration
 of $\tilde{\imath}^{-1}\tilde{\imath}_+{\Omega}_{\wtl{Y}}\spcheck$
 defined in a way similar to
 $F_p i^{-1}i_+{\cal L}$,
 then we get a map
\begin{align*}
\psi_p:\ 
&(\pi_K)_*\biggl(F_p\tilde{\imath}^{-1}
 \tilde{\imath}_+ {\Omega}_{\wtl{Y}}\spcheck
 \otimes_{\tilde{\imath}^{-1}{\cal O}_{\wtl{X}}} 
 \tilde{\imath}^{-1} \Bigl(
 \bigwedge^\bullet {\cal T}_{\wtl{X}/X}
 \otimes_{{\cal O}_{\wtl{X}}}
 \Omega_{\wtl{X}/X}
 \Bigr)\biggr)
 \otimes_{i^{-1}{\cal O}_X} {\cal V}\\
\to\ 
&(\pi_K)_*\biggl(F_p\tilde{\imath}^{-1}
 \tilde{\imath}_+ {\Omega}_{\wtl{Y}}\spcheck
 \otimes_{\tilde{\imath}^{-1}{\cal O}_{\wtl{X}}} 
 \tilde{\imath}^{-1} \Bigl(
 \bigwedge^\bullet {\cal T}_{\wtl{X}/X}
 \otimes_{{\cal O}_{\wtl{X}}}
 \Omega_{\wtl{X}/X}
 \Bigr)
 \otimes_{\pi_K^{-1}i^{-1}{\cal O}_{X}} 
 \pi_K^{-1}{\cal V} \biggr).
\end{align*}
It is enough to show that $\psi_p$ is an isomorphism for all $p\geq 0$
 because
 $\varinjlim_{p} F_p\tilde{\imath}^{-1}\tilde{\imath}_+ 
 {\Omega}_{\wtl{Y}}\spcheck
 \simeq \tilde{\imath}^{-1}\tilde{\imath}_+ {\Omega}_{\wtl{Y}}\spcheck$.
Since the ideal $\pi_K^{-1}(i^{-1}{\cal I}_Y)^{p+1}$
 of $\pi_K^{-1}i^{-1}{\cal O}_X$
 annihilates $F_p\tilde{\imath}^{-1}\tilde{\imath}_+ 
 {\Omega}_{\wtl{Y}}\spcheck$,
 we have
\begin{align*}
&(\pi_K)_*\biggl(F_p\tilde{\imath}^{-1}
 \tilde{\imath}_+ {\Omega}_{\wtl{Y}}\spcheck
 \otimes_{\tilde{\imath}^{-1}{\cal O}_{\wtl{X}}} 
 \tilde{\imath}^{-1} \Bigl(
 \bigwedge^\bullet {\cal T}_{\wtl{X}/X}
 \otimes_{{\cal O}_{\wtl{X}}}
 \Omega_{\wtl{X}/X}
 \Bigr)\biggr)
 \otimes_{i^{-1}{\cal O}_X} {\cal V}\\
\simeq\ 
&(\pi_K)_*\biggl(F_p\tilde{\imath}^{-1}
 \tilde{\imath}_+ {\Omega}_{\wtl{Y}}\spcheck
 \otimes_{\tilde{\imath}^{-1}{\cal O}_{\wtl{X}}} 
 \tilde{\imath}^{-1} \Bigl(
 \bigwedge^\bullet {\cal T}_{\wtl{X}/X}
 \otimes_{{\cal O}_{\wtl{X}}}
 \Omega_{\wtl{X}/X}
 \Bigr)\biggr)
 \otimes_{{\cal O}_{Y_{p+1}}} 
 ({\cal V}/(i^{-1}{\cal I}_Y)^{p+1}{\cal V}).
\end{align*}
By Definition \ref{de:assmod} (2), 
 ${\cal V}/(i^{-1}{\cal I}_Y)^{p+1}{\cal V}$ is a flat 
 ${\cal O}_{Y_{p+1}}$-module.
Hence the projection formula shows that $\psi_p$
 is an isomorphism and the claim is now verified.

The successive quotient of the filtration
\[ F_p{\cal M}:=
 F_p\tilde{\imath}^{-1}
 \tilde{\imath}_+ {\Omega}_{\wtl{Y}}\spcheck
 \otimes_{\tilde{\imath}^{-1}{\cal O}_{\wtl{X}}} 
 \tilde{\imath}^{-1} \Bigl(
 \bigwedge^d {\cal T}_{\wtl{X}/X}
 \otimes_{{\cal O}_{\wtl{X}}}
 \Omega_{\wtl{X}/X}
 \Bigr)
 \otimes_{\pi_K^{-1}i^{-1}{\cal O}_{X}} 
 \pi_K^{-1}{\cal V}\]
is 
\[(F_p\tilde{\imath}^{-1}\tilde{\imath}_+ {\Omega}_{\wtl{Y}}\spcheck/
 F_{p-1}\tilde{\imath}^{-1}
 \tilde{\imath}_+ {\Omega}_{\wtl{Y}}\spcheck)
 \otimes_{{\cal O}_{\wtl{Y}}} 
 \tilde{\imath}^* \Bigl(
 \bigwedge^d {\cal T}_{\wtl{X}/X}
 \otimes_{{\cal O}_{\wtl{X}}}
 \Omega_{\wtl{X}/X}
 \Bigr)
 \otimes_{{\cal O}_{\wtl{Y}}} 
 \pi_K^*({\cal V}/(i^{-1}{\cal I}_Y){\cal V}),\]
which is a quasi-coherent ${\cal O}_{\wtl{Y}}$-module.
Since $\wtl{Y}$ is affine, it follows that
${\rm H}^s(\wtl{Y}, F_p{\cal M}/F_{p-1}{\cal M}) =0$
for $s>0$.
Hence
 ${\rm H}^s(\wtl{Y}, F_p{\cal M}) =0$ and 
\[
{\rm H}^s\Bigl(\wtl{Y}, \tilde{\imath}^{-1}
 \tilde{\imath}_+ {\Omega}_{\wtl{Y}}\spcheck
 \otimes_{\tilde{\imath}^{-1}{\cal O}_{\wtl{X}}} 
 \tilde{\imath}^{-1} \Bigl(
 \bigwedge^d {\cal T}_{\wtl{X}/X}
 \otimes_{{\cal O}_{\wtl{X}}}
 \Omega_{\wtl{X}/X}
 \Bigr)
 \otimes_{\pi_K^{-1}i^{-1}{\cal O}_{X}} 
 \pi_K^{-1}{\cal V}\Bigr)=0
\]
for $s>0$.
By \eqref{isom1} and \eqref{projf}, we conclude that
\begin{align*}
&{\rm H}^s(Y, i^{-1}i_+{\cal L}\otimes_{i^{-1}{\cal O}_X} {\cal V})\\
\simeq \ 
 &{\rm H}^{s-u}
\Gamma \Bigl(\wtl{Y},\ 
 \tilde{\imath}^{-1}
 \tilde{\imath}_+ {\Omega}_{\wtl{Y}}\spcheck
 \otimes_{\tilde{\imath}^{-1}{\cal O}_{\wtl{X}}} 
 \tilde{\imath}^{-1} \Bigl(
 \bigwedge^\bullet {\cal T}_{\wtl{X}/X}
 \otimes_{{\cal O}_{\wtl{X}}}
 \Omega_{\wtl{X}/X}
 \Bigr)
 \otimes_{\pi_K^{-1}i^{-1}{\cal O}_X} 
 \pi_K^{-1} {\cal V} \Bigr).
\end{align*}
Since $ \tilde{\imath}^{-1} \tilde{\imath}_+ {\Omega}_{\wtl{Y}}\spcheck
 \otimes_{\tilde{\imath}^{-1}{\cal O}_{\wtl{X}}}
 \tilde{\imath}^{-1}\Omega_{\wtl{X}}
 \simeq {\cal O}_{\wtl{Y}} \otimes_{{\cal D}_{\wtl{Y}}}
 \tilde{\imath}^*{\cal D}_{\wtl{X}}$,
 we have 
\begin{align*}
& \tilde{\imath}^{-1}
 \tilde{\imath}_+ {\Omega}_{\wtl{Y}}\spcheck
 \otimes_{\tilde{\imath}^{-1}{\cal O}_{\wtl{X}}} 
 \tilde{\imath}^{-1} \Bigl(
 \bigwedge^\bullet {\cal T}_{\wtl{X}/X}
 \otimes_{{\cal O}_{\wtl{X}}}
 \Omega_{\wtl{X}/X}
 \Bigr)
 \otimes_{\pi_K^{-1}i^{-1}{\cal O}_X} 
 \pi_K^{-1} {\cal V} \\
\simeq\ 
&{\cal O}_{\wtl{Y}} \otimes_{{\cal D}_{\wtl{Y}}}
 \tilde{\imath}^*{\cal D}_{\wtl{X}}
 \otimes_{\tilde{\imath}^{-1}{\cal O}_{\wtl{X}}} 
 \tilde{\imath}^{-1} 
 \bigwedge^\bullet {\cal T}_{\wtl{X}/X}
 \otimes_{\pi_K^{-1}i^{-1}{\cal O}_X} 
 \pi_K^{-1} ({\cal V} \otimes_{i^{-1}{\cal O}_X} i^{-1}\Omega_X\spcheck).
\end{align*}
If we put
\begin{align*}
{\cal V}^{-d}:=\tilde{\imath}^{-1}\bigwedge^{d}{\cal T}_{\wtl{X}/X}
 \otimes_{\pi_K^{-1}i^{-1}{\cal O}_X}
 \pi_K^{-1}({\cal V} \otimes_{i^{-1}{\cal O}_X}
 i^{-1}\Omega_{X}\spcheck),
\end{align*}
then we obtain
\begin{align}
\label{shcohiso}
{\rm H}^s(Y, i^{-1}i_+{\cal L}\otimes_{i^{-1}{\cal O}_X} {\cal V})
\simeq 
 {\rm H}^{s-u}
\Gamma \bigl(\wtl{Y},\ 
{\cal O}_{\wtl{Y}} \otimes_{{\cal D}_{\wtl{Y}}}
 \tilde{\imath}^*{\cal D}_{\wtl{X}}
 \otimes_{\tilde{\imath}^{-1}{\cal O}_{\wtl{X}}} 
 {\cal V}^\bullet).
\end{align}
The boundary map
\[\partial:{\cal O}_{\wtl{Y}} \otimes_{{\cal D}_{\wtl{Y}}}
 \tilde{\imath}^*{\cal D}_{\wtl{X}}
 \otimes_{\tilde{\imath}^{-1}{\cal O}_{\wtl{X}}} 
 {\cal V}^{-d} \to {\cal O}_{\wtl{Y}} \otimes_{{\cal D}_{\wtl{Y}}}
 \tilde{\imath}^*{\cal D}_{\wtl{X}}
 \otimes_{\tilde{\imath}^{-1}{\cal O}_{\wtl{X}}} 
 {\cal V}^{-d+1}\]
is given by
\begin{align*}
 &f\otimes D \otimes \wtl{\xi_1} \wedge \cdots\wedge \wtl{\xi_d} \otimes v\\
&\mapsto
\sum_{i=1}^d (-1)^{i+1} f\otimes D \wtl{\xi_i}
 \otimes \wtl{\xi_1} \wedge \cdots \wedge \widehat{\wtl{\xi_i}}
 \wedge \cdots \wedge \wtl{\xi_d} \otimes v\\
&+\sum_{1\leq i<j\leq d} (-1)^{i+j} f\otimes D \otimes
 [\wtl{\xi_i}, \wtl{\xi_j}] \wedge \wtl{\xi_1} \wedge
 \cdots \wedge \widehat{\wtl{\xi_i}} \wedge \cdots \wedge
 \widehat{\wtl{\xi_j}} \wedge \cdots \wedge \wtl{\xi_d} \otimes v,
\end{align*}
where $f\in{\cal O}_{\wtl{Y}}$,
 $D\in \tilde{\imath}^*{\cal D}_{\wtl{X}}$,
 $\wtl{\xi_1},\dots,\wtl{\xi_d}\in\tilde{\imath}^{-1}{\cal T}_{\wtl{X}/X}$,
 and
 $v\in \pi_K^{-1}({\cal V}\otimes_{i^{-1}{\cal O}_X} i^{-1}\Omega_X\spcheck)$.

The right side of \eqref{shcohiso}
 can be computed by using the following lemma.
\begin{lem}
\label{affiso}
Let $V'$ be an $L$-module, or equivalently an $(\frak{l}, L)$-module.
Let ${\cal V}'$ be an 
 $\tilde{\imath}^{-1}\wtl{\frak{g}}_{\wtl{X}}$-module
 associated with $V'$.
Then 
\[\Gamma(\wtl{Y},  {\cal O}_{\wtl{Y}}
 \otimes_{{\cal D}_{\wtl{Y}}}
 \tilde{\imath}^*{\cal D}_{\wtl{X}}
 \otimes_{\tilde{\imath}^{-1}{\cal O}_{\wtl{X}}} 
 {\cal V}' )
\simeq R(\frak{g},K)\otimes_{R(L)} V'.
\]
\end{lem}
\begin{proof}
The proof is similar to that of \cite[Lemma 3.4]{Os11}.

Using the right $\tilde{\imath}^{-1}{\cal D}_{\wtl{X}}$-module
 structure of
 $\tilde{\imath}^*{\cal D}_{\wtl{X}}$,
 we can define a $\frak{g}$-action $\rho$ on
 the sheaf $\tilde{\imath}^*{\cal D}_{\wtl{X}}
 \otimes_{\tilde{\imath}^{-1}{\cal O}_{\wtl{X}}} {\cal V}'$ by 
\[
\rho(\xi)(D\otimes v):= -D\xi_{\wtl{X}}\otimes v+D\otimes \xi v
\]
for $\xi\in\frak{g}$, $D\in \tilde{\imath}^*{\cal D}_{\wtl{X}}$, and
 $v\in {\cal V}'$.
Moreover,
 the sheaf $\tilde{\imath}^*{\cal D}_{\wtl{X}}
  \otimes_{\tilde{\imath}^{-1}{\cal O}_{\wtl{X}}} {\cal V}'$
 is $K$-equivariant.
We denote this $K$-action and also its infinitesimal
 $\frak{k}$-action by $\nu$.
Definition \ref{de:assmod} (4) implies that 
 the $\frak{k}$-action $\nu$ is given by
\[
\nu(\eta)(D\otimes v)= \eta_{\wtl{Y}} D \otimes v - D\eta_{\wtl{X}} \otimes v
+D\otimes \eta v
\]
for $\eta\in\frak{k}$.
Here, $\eta_{\wtl{Y}} D$ and $D\eta_{\wtl{X}}$ are defined by
 the $({\cal D}_{\wtl{Y}}, \tilde{\imath}^{-1}{\cal D}_{\wtl{X}})$-bimodule
 structure on $\tilde{\imath}^*{\cal D}_{\wtl{X}}$.
Then it follows from Definition \ref{de:assmod} (3) that 
 $\Gamma(\wtl{Y}, \tilde{\imath}^*{\cal D}_{\wtl{X}}
 \otimes_{\tilde{\imath}^{-1}{\cal O}_{\wtl{X}}} {\cal V}')$
 is a weak $(\frak{g},K)$-module in the sense of \cite{BeLu}, namely,
\begin{align*}
\nu(k)\rho(\xi)\nu(k^{-1})=\rho(\Ad (i(k))\xi)
\end{align*}
 for $k\in K$ and $\xi\in\frak{g}$.
Put $\omega(\eta):=\nu(\eta)-\rho(\eta)$
 for $\eta\in\frak{k}$.
Then $\omega(\eta)$ is given by
\[
\omega(\eta)(D\otimes v)= \eta_{\wtl{Y}} D\otimes v.
\]
Since $\wtl{Y}$ is an affine variety, 
 $\Gamma(\wtl{Y}, {\cal D}_{\wtl{Y}})$ is generated by $U(\frak{k})$
 and ${\cal O}(\wtl{Y})$ as an algebra.
Therefore,
\begin{align*}
&\Gamma(\wtl{Y}, {\cal O}_{\wtl{Y}}
 \otimes_{{\cal D}_{\wtl{Y}}} \tilde{\imath}^*{\cal D}_{\wtl{X}}
 \otimes_{\tilde{\imath}^{-1}{\cal O}_{\wtl{X}}}
 {\cal V}') \\
\simeq\ &{\cal O}(\wtl{Y})
 \otimes_{\Gamma(\wtl{Y}, {\cal D}_{\wtl{Y}})}
 \Gamma(\wtl{Y}, \tilde{\imath}^*{\cal D}_{\wtl{X}}
 \otimes_{\tilde{\imath}^{-1}{\cal O}_{\wtl{X}}}
 {\cal V}')\\
\simeq\ &\Gamma(\wtl{Y}, \tilde{\imath}^*{\cal D}_{\wtl{X}}
 \otimes_{\tilde{\imath}^{-1}{\cal O}_{\wtl{X}}}
  {\cal V}')
 \,/\,\omega(\frak{k})
 \Gamma(\wtl{Y}, {\tilde{\imath}}^*{\cal D}_{\wtl{X}}
 \otimes_{{\tilde{\imath}}^{-1}{\cal O}_{\wtl{X}}} 
 {\cal V}').
\end{align*}

Let $e\in K$ be the identity element. 
Write $o:=eL\in \wtl{Y}$ for the base point
 and $i_{o}:\{o\}\to \wtl{Y}$ for the inclusion map.
Let ${\cal I}_o$ be the maximal ideal of ${\cal O}_{\wtl{Y}}$
 corresponding to $o$.
The fiber of
 $\tilde{\imath}^*{\cal D}_{\wtl{X}}
 \otimes_{\tilde{\imath}^{-1}{\cal O}_{\wtl{X}}} {\cal V}'$
 at $o$ is given by
\begin{align*}
W:=\ &i_{o}^*
 (\tilde{\imath}^*{\cal D}_{\wtl{X}}
 \otimes_{\tilde{\imath}^{-1}{\cal O}_{\wtl{X}}} {\cal V}')\\
 \simeq\ 
&\Gamma(\wtl{Y}, \tilde{\imath}^*{\cal D}_{\wtl{X}}
 \otimes_{\tilde{\imath}^{-1}{\cal O}_{\wtl{X}}} {\cal V}')
 \,/\, {\cal I}_o(\wtl{Y})
 \Gamma(\wtl{Y}, \tilde{\imath}^*{\cal D}_{\wtl{X}}
 \otimes_{\tilde{\imath}^{-1}{\cal O}_{\wtl{X}}}{\cal V}').
\end{align*}
The actions $\rho$ and $\nu$ on
 $\tilde{\imath}^*{\cal D}_{\wtl{X}}
 \otimes_{\tilde{\imath}^{-1}{\cal O}_{\wtl{X}}}{\cal V}'$
 induce a $\frak{g}$-action and an
 $L$-action on $W$.
With these actions, $W$ becomes a $(\frak{g},L)$-module
 and there is an isomorphism
\begin{align}
\label{varphiisom}
\varphi : U(\frak{g})\otimes_{U(\frak{l})} V'
 \xrightarrow{\sim} W.
\end{align}
This can be proved 
 by using \cite[Lemma 3.3]{Os11}
 and Definition~\ref{de:assmod}
 (see the proof of \cite[Lemma 3.4]{Os11}).
Hence we have 
\[
\Gamma(\wtl{Y},
 \tilde{\imath}^*{\cal D}_{\wtl{X}}
 \otimes_{\tilde{\imath}^{-1}{\cal O}_{\wtl{X}}}
  {\cal V}')\simeq
R(K)\otimes_{R(L)} (U(\frak{g})\otimes_{U(\frak{l})} V').
\]

The rest is the same as \cite[Lemma 3.4]{Os11}.
%\qed
\end{proof}

Returning to the proof of Theorem~\ref{loccoh}, 
 let us compute the cohomological induction
 $(P_{\frak{h},L}^{\frak{g},K})_s
 (V\otimes \bigwedge^{\rm top}(\frak{g}/\frak{h}))$
 by using the standard resolution (\cite[\S II.7]{KnVo}).
The standard resolution 
 is a projective resolution of  
 the $(\frak{h},L)$-module
 $V\otimes \bigwedge^{\rm top}(\frak{g}/\frak{h})$
 given by the complex 
\begin{align*}
&U(\frak{h})\otimes_{U(\frak{l})}
\Bigl(\bigwedge^{\bullet}(\frak{h}/\frak{l})
 \otimes V\otimes 
 \bigwedge^{\rm top}(\frak{g}/\frak{h})\Bigr),
\end{align*}
where the boundary map 
\[\partial' : 
U(\frak{h})\otimes_{U(\frak{l})}\Bigl(\bigwedge^{d}(\frak{h}/\frak{l})
 \otimes V\otimes 
 \bigwedge^{\rm top}(\frak{g}/\frak{h})\Bigr)
\to
U(\frak{h})\otimes_{U(\frak{l})}\Bigl(\bigwedge^{d-1}(\frak{h}/\frak{l})
 \otimes V\otimes 
 \bigwedge^{\rm top}(\frak{g}/\frak{h})\Bigr)
\]
 is 
\begin{align*}
& D\otimes \ovl{\xi_1}\wedge\cdots\wedge\ovl{\xi_d}
  \otimes v\\
\mapsto\ 
&\sum_{i=1}^d (-1)^{i+1}(D\xi_i\otimes
 \ovl{\xi_1}\wedge\cdots\wedge\widehat{\ovl{\xi_i}}
 \wedge\cdots\wedge\ovl{\xi_d}\otimes v
 -D\otimes
 \ovl{\xi_1}\wedge\cdots\wedge\widehat{\ovl{\xi_i}}
 \wedge\cdots\wedge\ovl{\xi_d}\otimes \xi_i v)\\
+\ 
&\sum_{1\leq i<j\leq d} (-1)^{i+j}D\otimes
 \ovl{[\xi_i,\xi_j]}\wedge\ovl{\xi_1}\wedge\cdots
 \wedge\widehat{\ovl{\xi_i}}\wedge\cdots\wedge\widehat{\ovl{\xi_j}}
 \wedge\cdots\wedge\ovl{\xi_d}\otimes v
\end{align*}
for $D\in U(\frak{h})$, $\xi_1,\cdots,\xi_d\in\frak{h},$
 and $v\in V\otimes \bigwedge^{\rm top}(\frak{g}/\frak{h})$.
Therefore,
\begin{align}
\label{cohiso}
(P_{\frak{h},L}^{\frak{g},K})_{u-s}
\Bigl(V\otimes\bigwedge^{\rm top}(\frak{g}/\frak{h})\Bigr)
&\simeq
{\rm H}^{s-u}
P_{\frak{h},L}^{\frak{g},K}
 \biggl(
 U(\frak{h})\otimes_{U(\frak{l})}\Bigl(\bigwedge^{\bullet}(\frak{h}/\frak{l})
 \otimes V\otimes 
 \bigwedge^{\rm top}(\frak{g}/\frak{h})\Bigr)\biggr)\\ \nonumber
&\simeq
{\rm H}^{s-u}
R(\frak{g},K)\otimes_{R(L)}
\Bigl(\bigwedge^{\bullet}(\frak{h}/\frak{l})
 \otimes V\otimes 
 \bigwedge^{\rm top}(\frak{g}/\frak{h})\Bigr),
\end{align}
where the boundary map 
\[\partial':
 R(\frak{g},K)\otimes_{R(L)}
\Bigl(\bigwedge^{d}(\frak{h}/\frak{l})
 \otimes V\otimes 
 \bigwedge^{\rm top}(\frak{g}/\frak{h})\Bigr)\to
 R(\frak{g},K)\otimes_{R(L)}
\Bigl(\bigwedge^{d-1}(\frak{h}/\frak{l})
 \otimes V\otimes 
 \bigwedge^{\rm top}(\frak{g}/\frak{h})\Bigr)\]
is given by
\begin{align*}
& D\otimes \ovl{\xi_1}\wedge\cdots\wedge\ovl{\xi_d}
  \otimes v\\
\mapsto\ 
&\sum_{i=1}^d (-1)^{i+1}(D\xi_i\otimes
 \ovl{\xi_1}\wedge\cdots\wedge\widehat{\ovl{\xi_i}}
 \wedge\cdots\wedge\ovl{\xi_d}\otimes v
 -D\otimes
 \ovl{\xi_1}\wedge\cdots\wedge\widehat{\ovl{\xi_i}}
 \wedge\cdots\wedge\ovl{\xi_d}\otimes \xi_i v)\\
+\ 
&\sum_{1\leq i<j\leq d} (-1)^{i+j}D\otimes
 \ovl{[\xi_i,\xi_j]}\wedge\ovl{\xi_1}\wedge\cdots
 \wedge\widehat{\ovl{\xi_i}}\wedge\cdots\wedge\widehat{\ovl{\xi_j}}
 \wedge\cdots\wedge\ovl{\xi_d}\otimes v
\end{align*}
for $D\in R(\frak{g},K)$, $\xi_1,\cdots,\xi_d\in\frak{h},$
 and $v\in V\otimes \bigwedge^{\rm top}(\frak{g}/\frak{h})$.

Put 
\begin{align*}
V^{-d}:=\bigwedge^{d}(\frak{h}/\frak{l})\otimes V\otimes
 \bigwedge^{\rm top}(\frak{g}/\frak{h})
\end{align*}
 for simplicity.
We identify the fiber of ${\cal T}_{\wtl{X}/X}$ with
 $\frak{h}/\frak{l}$ in the following way:
 if a vector field $\wtl{\xi}\in{\cal T}_{\wtl{X}/X}$ equals $-\xi_{\wtl{X}}$
 at the base point $eL \in \wtl{X}$ for $\xi\in\frak{h}$,
 then $\wtl{\xi}$ takes the value $\ovl{\xi}\in\frak{h}/\frak{l}$
 at $e \in G$.
Similarly, the fiber of $\Omega_{\wtl{X}/X}\spcheck$ is
 identified with $\bigwedge^{\rm top}(\frak{g}/\frak{h})$.
Then ${\cal V}^{-d}$ is associated with $V^{-d}$ by
 Example~\ref{vectbdle} and Example~\ref{tensor}.
From \eqref{shcohiso} and \eqref{cohiso}
 it is enough to show that
 the isomorphisms $\varphi$ given in Lemma \ref{affiso} for
 $V'=V^{-d}$, 
 $0\leq d\leq {\rm dim}(\frak{h}/\frak{l})$
 commute with the boundary maps, that is, the diagram
\begin{align*}
\xymatrix{
R(\frak{g}, K)\otimes_{R(L)} V^{-d} 
\ar[r]^*+{\partial'} \ar[d]
& R(\frak{g}, K)\otimes_{R(L)} V^{-d+1} \ar[d]\\
\Gamma(\wtl{Y},  {\cal O}_{\wtl{Y}}
 \otimes_{{\cal D}_{\wtl{Y}}}
 \tilde{\imath}^*{\cal D}_{\wtl{X}}
 \otimes_{\tilde{\imath}^{-1}{\cal O}_{\wtl{X}}} 
 {\cal V}^{-d}) \ar[r]^>>>>>*+{\partial}   
& \Gamma(\wtl{Y},  {\cal O}_{\wtl{Y}}
 \otimes_{{\cal D}_{\wtl{Y}}}
 \tilde{\imath}^*{\cal D}_{\wtl{X}}
 \otimes_{\tilde{\imath}^{-1}{\cal O}_{\wtl{X}}} 
 {\cal V}^{-d+1}) 
}
\end{align*}
commutes.
In view of the proof of Lemma \ref{affiso}, the above diagram
 is obtained by applying the functor $P_{\frak{g},L}^{\frak{g},K}$ to
\begin{align}
\label{diagram}
\xymatrix{
U(\frak{g})\otimes_{U(\frak{l})} V^{-d} 
\ar[r]^*+{\partial'} \ar[d]_*+{\varphi^{d}} 
& U(\frak{g})\otimes_{U(\frak{l})} V^{-d+1}
 \ar[d]^*+{\varphi^{d-1}}\\
 i_o^*(\tilde{\imath}^*{\cal D}_{\wtl{X}}
 \otimes_{\tilde{\imath}^{-1}{\cal O}_{\wtl{X}}}
 {\cal V}^{-d}) \ar[r]^>>>>>*+{\partial}   
& i_o^*(\tilde{\imath}^*{\cal D}_{\wtl{X}}
 \otimes_{\tilde{\imath}^{-1}{\cal O}_{\wtl{X}}}
 {\cal V}^{-d+1}),
}
\end{align}
where $\varphi^d$ is the map $\varphi$ of \eqref{varphiisom} for $V'=V^{-d}$.
Therefore, it suffices to show that the diagram \eqref{diagram} commutes.

To see this, we use the following notation.
A section 
$f\in \tilde{\imath}^*{\cal D}_{\wtl{X}}
 \otimes_{\tilde{\imath}^{-1}{\cal O}_{\wtl{X}}} {\cal V}^{-d}$
 defines a section of $i_o^*(\tilde{\imath}^*{\cal D}_{\wtl{X}}
 \otimes_{\tilde{\imath}^{-1}{\cal O}_{\wtl{X}}} {\cal V}^{-d})$
 and hence defines an element of $U(\frak{g})\otimes_{U(\frak{l})} V^{-d}$ via
 the isomorphism $\varphi^d$.
We write $i_o^*f\in U(\frak{g})\otimes_{U(\frak{l})} V^{-d}$ for this element.
Put $Z:=H/L$ and write $i_Z:Z\to \wtl{X}$ for the inclusion map.
Then $i_Z(Z)=\pi^{-1}(\{o\})$ and there is a canonical isomorphism
 $i_Z^*{\cal T}_{\wtl{X}/X}\simeq {\cal T}_Z$.
For $\xi_1,\cdots,\xi_d\in\frak{h}$ and
$v\in V\otimes \bigwedge^{\rm top}(\frak{g}/\frak{h})$, 
 put
\[m:=\ovl{\xi_1}\wedge\cdots\wedge\ovl{\xi_d}\otimes v\,\in\, V^{-d}.\]
We will choose sections
 $\wtl{\xi_i} \in \tilde{\imath}^{-1}{\cal T}_{\wtl{X}/X}$
 and
 $\wtl{v}\in
 \pi_K^{-1}({\cal V}\otimes_{i^{-1}{\cal O}_X}i^{-1}\Omega_X\spcheck)$
 on a neighborhood of the base point $o\in \wtl{Y}$ in the following way.
Take $\wtl{\xi_i}\in {\cal T}_{\wtl{X}/X}$ 
 such that $\wtl{\xi_i}|_{Z}
 \in i_Z^*{\cal T}_{\wtl{X}/X}$
 corresponds to $-(\xi_i)_{Z}$.
Then it gives a section of $\tilde{\imath}^{-1}{\cal T}_{\wtl{X}/X}$,
 which we denote by the same letter $\wtl{\xi_i}$.
We take a section $\wtl{v}\in
 \pi_K^{-1}({\cal V} \otimes_{i^{-1}{\cal O}_X} i^{-1}\Omega_X\spcheck)$
 on a neighborhood of $o$ 
 such that $i_o^*\wtl{v}$ corresponds to $v$.
Define a section $\wtl{m}\in {\cal V}^{-d}$
 in a neighborhood of $o$ as
\begin{align*}
\wtl{m}:=\wtl{\xi_1}\wedge\cdots\wedge\wtl{\xi_d}
 \otimes \wtl{v}
 \,\in\, {\cal V}^{-d}.
\end{align*}
Then the element $\varphi^d(1\otimes m)$ is represented by the section
\[1\otimes \wtl{m}\in
\tilde{\imath} {\,}^*{\cal D}_{\wtl{X}} 
 \otimes_{\tilde{\imath}^{-1}{\cal O}_{\wtl{X}}}
 {\cal V}^{-d},
\]
in other words, $i_o^* (1\otimes \wtl{m})=1\otimes m$.

We have
\begin{align*}
&\partial(1\otimes \wtl{m})\\
=\ 
 &\sum_{i=1}^d(-1)^{i+1} (\wtl{\xi_i} \otimes 
 \wtl{\xi_1}\wedge\cdots\wedge\widehat{\wtl{\xi_i}}\wedge\cdots\wedge
 \wtl{\xi_d}\otimes \wtl{v}) \\
 +\ &\sum_{1\leq i<j\leq d}(-1)^{i+j}\, 
 \bigl( 1\otimes [\wtl{\xi_i},\wtl{\xi_j}]\wedge
 \wtl{\xi_1}\wedge\cdots
 \wedge\widehat{\wtl{\xi_i}}\wedge\cdots\wedge\widehat{\wtl{\xi_j}}\wedge
 \cdots\wedge\wtl{\xi_d}\otimes \wtl{v}\bigr)
\end{align*}
and
\begin{align*}
&\partial'(1\otimes m)\\
=\ & \sum_{i=1}^d (-1)^{i+1} \bigl(\xi_i \otimes
 \ovl{\xi_1}\wedge\cdots\wedge\widehat{\ovl{\xi_i}}
 \wedge\cdots\wedge\ovl{\xi_d}
 \otimes v
- 1 \otimes
 \ovl{\xi_1}\wedge\cdots\wedge\widehat{\ovl{\xi_i}}
 \wedge\cdots\wedge\ovl{\xi_d}
 \otimes \xi_i v\bigr) \\
+\ &\sum_{1\leq i<j\leq d}(-1)^{i+j}\, 
 \bigl( 1 \otimes \ovl{[\xi_i,\xi_j]}\wedge
 \ovl{\xi_1}\wedge\cdots
 \wedge\widehat{\ovl{\xi_i}}\wedge\cdots\wedge
 \widehat{\ovl{\xi_j}}\wedge
 \cdots\wedge\ovl{\xi_d} \otimes v \bigr).
\end{align*}
Since $\wtl{\xi_i}|_{Z}$ corresponds to $-(\xi_i)_{Z}$, 
 the vector fields
 $\wtl{\xi_i}$ and $(\xi_i)_{\wtl{X}}$ have the relation 
 $\wtl{\xi_i} =-(\xi_i)_{\wtl{X}}$ at $o$.
Recall that the $\frak{g}$-action on
 ${\cal T}_{\wtl{X}/X}$
 is defined as the differential of the $G$-equivariant structure on it.
Hence our choice implies that $\xi_i\cdot\wtl{\xi_j}|_{Z}= -([\xi_i,\xi_j])_Z$.
As a result,  
\begin{align*}
&i_o^*\bigl(\wtl{\xi_i} \otimes 
 \wtl{\xi_1}\wedge\cdots\wedge\widehat{\wtl{\xi_i}}\wedge\cdots\wedge
 \wtl{\xi_d}\otimes \wtl{v}
 \bigr)\\ \nonumber
=\ &i_o^*
 \bigl(\rho(\xi_i)(1\otimes \wtl{\xi_1}\wedge\cdots\wedge\widehat{\wtl{\xi_i}}
 \wedge\cdots\wedge\wtl{\xi_d}\otimes \wtl{v})\bigr)
- i_o^*(1\otimes \wtl{\xi_1}\wedge\cdots\wedge\widehat{\wtl{\xi_i}}
 \wedge\cdots\wedge\wtl{\xi_d}\otimes \xi_i\wtl{v}) \\ \nonumber
& -\sum_{1\leq i<j\leq d} i_o^*\bigl(1\otimes 
 \wtl{\xi_1}\wedge\cdots\wedge\widehat{\wtl{\xi_i}}\wedge\cdots\wedge
 \wtl{\xi}_{j-1}\wedge (\xi_i \cdot \wtl{\xi_j}) \wedge
 \wtl{\xi}_{j+1}\wedge \cdots\wedge\wtl{\xi_d} \otimes \wtl{v}) \\ \nonumber
& -\sum_{1\leq j < i \leq d} i_o^*\bigl(1\otimes 
 \wtl{\xi_1}\wedge\cdots\wedge
 \wtl{\xi}_{j-1}\wedge
 (\xi_i\cdot \wtl{\xi_j}) \wedge \wtl{\xi}_{j+1}\wedge
 \cdots\wedge\widehat{\wtl{\xi_i}}\wedge\cdots
 \wedge\wtl{\xi_d}
 \otimes \wtl{v}) \\ \nonumber
=\ & \xi_i\otimes
 \ovl{\xi_1}\wedge\cdots\wedge\widehat{\ovl{\xi_i}}
 \wedge\cdots\wedge\ovl{\xi_d}
 \otimes v
 - 1\otimes \ovl{\xi_1}\wedge\cdots\wedge\widehat{\ovl{\xi_i}}
 \wedge\cdots\wedge\ovl{\xi_d}
 \otimes \xi_i v \\
& +\sum_{1\leq i<j\leq d} (-1)^{j+1}
 \bigl(1\otimes \ovl{[\xi_i, \xi_j]}
 \wedge \ovl{\xi_1}\wedge\cdots\wedge\widehat{\ovl{\xi_i}}\wedge\cdots\wedge
 \widehat{\ovl{\xi_j}}\wedge \cdots\wedge \ovl{\xi_d} \otimes v) \\ \nonumber
& +\sum_{1\leq j < i\leq d} (-1)^{j} \bigl(1\otimes \ovl{[\xi_i, \xi_j]}
 \wedge \ovl{\xi_1}\wedge\cdots\wedge\widehat{\ovl{\xi_j}}\wedge\cdots\wedge
 \widehat{\ovl{\xi_i}}\wedge \cdots\wedge \ovl{\xi_d} \otimes v).
\end{align*}
Moreover, $[\wtl{\xi_i},\wtl{\xi_j}]|_Z$ corresponds to
 $[-(\xi_i)_{Z}, -(\xi_j)_{Z}] =([\xi_i,\xi_j])_{Z}$.
Hence 
\begin{align*}
 i_o^*&\bigl(1\otimes [\wtl{\xi_i},\wtl{\xi_j}]\wedge
 \wtl{\xi_1}\wedge\cdots
 \wedge\widehat{\wtl{\xi_i}}\wedge\cdots\wedge\widehat{\wtl{\xi_j}}\wedge
 \cdots\wedge\wtl{\xi_d}\otimes \wtl{v}\bigr)\\
=\ &-1\otimes \ovl{[\xi_i, \xi_j]}\wedge
 \ovl{\xi_1}\wedge\cdots
 \wedge\widehat{\ovl{\xi_i}}\wedge\cdots\wedge\widehat{\ovl{\xi_j}}\wedge
 \cdots\wedge \ovl{\xi_d}\otimes v.
\end{align*}
We thus conclude that
\begin{align*}
&(\varphi^{d-1})^{-1}\circ \partial\circ \varphi^d(1\otimes m)\\
=\ &i_o^*(\partial(1\otimes \wtl{m}))\\
=\ &i_o^*\biggl( 
\sum_{i=1}^d (-1)^{i+1}(\wtl{\xi_i}\otimes \wtl{\xi_1}\wedge
 \cdots\wedge\widehat{\wtl{\xi_i}}\wedge\cdots\wedge
 \wtl{\xi_d}\otimes \wtl{v})\\
&\qquad +\sum_{1\leq i<j\leq d}(-1)^{i+j}(1\otimes [\wtl{\xi_i},\wtl{\xi_j}]
 \wedge\wtl{\xi_1}\wedge\cdots\wedge
 \widehat{\wtl{\xi_i}}\wedge\cdots\wedge
 \widehat{\wtl{\xi_j}}\wedge\cdots\wedge
 \wtl{\xi_d}\otimes \wtl{v})
\biggr)\\
=\ &\sum_{i=1}^d (-1)^{i+1} \biggl( \xi_i\otimes
 \ovl{\xi_1}\wedge\cdots\wedge
 \widehat{\ovl{\xi_i}}\wedge\cdots\wedge\ovl{\xi_d}
 \otimes v
 - 1\otimes \ovl{\xi_1}\wedge\cdots\wedge
 \widehat{\ovl{\xi_i}}\wedge\cdots\wedge\ovl{\xi_d}
 \otimes \xi_i v \\
& \qquad +\sum_{1\leq i<j\leq d} (-1)^{j+1} \bigl(1\otimes \ovl{[\xi_i, \xi_j]}
 \wedge \ovl{\xi_1}\wedge\cdots\wedge\widehat{\ovl{\xi_i}}
 \wedge\cdots\wedge
 \widehat{\ovl{\xi_j}}\wedge \cdots\wedge \ovl{\xi_d} \otimes v) \\ \nonumber
& \qquad +\sum_{1\leq j<i\leq d} (-1)^{j}
 \bigl(1\otimes \ovl{[\xi_i, \xi_j]}
 \wedge \ovl{\xi_1}\wedge\cdots\wedge\widehat{\ovl{\xi_j}}
 \wedge\cdots\wedge
 \widehat{\ovl{\xi_i}}\wedge \cdots\wedge \ovl{\xi_d} \otimes v)\biggr)\\
& +\sum_{1\leq i<j\leq d} (-1)^{i+j+1}(1\otimes \ovl{[\xi_i, \xi_j]}\wedge
 \ovl{\xi_1}\wedge\cdots
 \wedge\widehat{\ovl{\xi_i}}\wedge\cdots\wedge\widehat{\ovl{\xi_j}}\wedge
 \cdots\wedge \ovl{\xi_d}\otimes v) \\
=\ &\sum_{i=1}^d (-1)^{i+1} (\xi_i\otimes
 \ovl{\xi_1}\wedge\cdots\wedge\widehat{\ovl{\xi_i}}
 \wedge\cdots\wedge\ovl{\xi_d}
 \otimes v
 - 1\otimes \ovl{\xi_1}\wedge\cdots\wedge\widehat{\ovl{\xi_i}}
 \wedge\cdots\wedge\ovl{\xi_d}
 \otimes \xi_i v) \\
& +\sum_{1\leq i<j\leq d} (-1)^{i+j} \bigl(1\otimes \ovl{[\xi_i, \xi_j]}
 \wedge \ovl{\xi_1}\wedge\cdots\wedge
 \widehat{\ovl{\xi_i}}\wedge\cdots\wedge
 \widehat{\ovl{\xi_j}}\wedge \cdots\wedge \ovl{\xi_d} \otimes v) \\ \nonumber
=\ &\partial' (1\otimes m).
\end{align*}
Since $\partial$, $\partial'$ and $\varphi^d$ commute with $\frak{g}$-actions,
\[\partial(\varphi^d(D\otimes m))=D\partial(\varphi^d (1\otimes m))
 =D\varphi^{d-1}(\partial'(1\otimes m))
 =\varphi^{d-1}(\partial'(D\otimes m))\]
 for $D\in U(\frak{g})$.
Consequently, the diagram \eqref{diagram} commutes 
 and the proof of the theorem is complete.
%\qed
\end{proof}

\section{Construction of modules}
\label{sec:const}
In this section,
 we will construct an $i^{-1}\wtl{\frak{g}}_X$-module
 ${\cal V}$ associated with a $(\frak{h}, M)$-module $V$,
 which can be used in Section \ref{sec:loc} for the realization of
 cohomologically induced modules.

Let ${\cal V}_Y$ be the $K$-equivariant quasi-coherent ${\cal O}_Y$-module
 with typical fiber the $M$-module $V$.
Let 
$p: {\cal O}_X\otimes_\bb{C} \frak{g}\to {\cal T}_X$
 be the map
 given by $f\otimes \xi\mapsto f\xi_X$
 and put ${\cal H}:=\ker\, p$.  
The ${\cal O}_X$-module ${\cal H}$ is 
 $G$-equivariant with typical fiber $\frak{h}$.
Hence a section $\xi \in {\cal H}$ 
 is identified with a $\frak{h}$-valued regular function on
 a subset of $G$ satisfying $\xi(gh)={\rm Ad}(h^{-1})(\xi(g))$ for
 $h\in H$.
Let $\xi,\xi' \in {\cal H}$.
By regarding $\wtl{\frak{g}}_X={\cal O}_X\otimes_\bb{C} \frak{g}$
 as a submodule of
 $U(\wtl{\frak{g}}_X)={\cal O}_X\otimes_\bb{C} U(\frak{g})$,
 we have
 $[\xi,\xi']:=\xi\xi'-\xi'\xi \in {\cal H}$
 and $[\xi,\xi'](g)=[\xi(g),\xi'(g)]$ with the identification above.
If we write $\xi=\sum_{i} f_i \otimes \xi_i$ for $f_i\in{\cal O}_X$
 and $\xi_i\in\frak{g}$, then 
 $\xi(g)=\sum_i f_i(g)\Ad(g^{-1})(\xi_i)$.

Let 
 ${\cal A}$ be the subalgebra of
 $i^{-1}U(\wtl{\frak{g}}_X)=i^{-1}{\cal O}_X \otimes U(\frak{g})$
 generated by 
 $i^{-1}{\cal H}$, $1 \otimes \frak{k}$, and $i^{-1}{\cal O}_X \otimes 1$.
We view $i^{-1}U(\wtl{\frak{g}}_X)$ as an $i^{-1}{\cal O}_X$-module
 and consider the inverse image
 ${\cal O}_Y \otimes_{i^{-1}{\cal O}_X} i^{-1}U(\wtl{\frak{g}}_X)
 (\simeq {\cal O}_Y \otimes U(\frak{g}))$ of $U(\wtl{\frak{g}}_X)$.
Let $\ovl{\cal A}$
 be the image of the map
 ${\cal O}_Y \otimes_{i^{-1}{\cal O}_X} {\cal A} \to 
{\cal O}_Y \otimes_{i^{-1}{\cal O}_X} i^{-1}U(\wtl{\frak{g}}_X)$
 so that $\ovl{\cal A}\simeq
 {\cal A}/
\bigl({\cal A} \cap \bigl(i^{-1}{\cal I}_Y \otimes U(\frak{g})\bigr)\bigr)$.
Since ${\cal A} \cdot (i^{-1}{\cal I}_Y\otimes U(\frak{g})) 
 \subset i^{-1}{\cal I}_Y\otimes U(\frak{g})$
 in the algebra $i^{-1}U(\wtl{\frak{g}}_X)$,
 the algebra structure of ${\cal A}$ induces
 that of $\ovl{\cal A}$, and 
 ${\cal O}_Y \otimes_{i^{-1}{\cal O}_X} i^{-1}U(\wtl{\frak{g}}_X)$
 becomes a left $\ovl{\cal A}$-module.

We give a left $\ovl{\cal A}$-module structure on ${\cal V}_Y$
 in the following way.
We view a local section of ${\cal V}_Y$ as a $V$-valued regular
 function on a subset of $K$ and 
define a $(1\otimes i^{-1}{\cal H})$-action
 and an $({\cal O}_Y\otimes 1)$-action by
\begin{align*}
&((1\otimes\xi) v)(k)=\xi(i(k))v(k), \\
&(f\otimes 1)v=fv
\end{align*}
for $\xi\in i^{-1}{\cal H}$, $v\in {\cal V}_Y$, $f\in {\cal O}_Y$,
 and $k\in K$;
define a $(1\otimes \frak{k})$-action on ${\cal V}_Y$ by differentiating
 the $K$-action on ${\cal V}_Y$.
These actions are compatible in the following sense:
if $f_i\in i^{-1}{\cal O}_X$, $\eta_i\in\frak{k}$
 and $\xi \in i^{-1}{\cal H}$ satisfy 
\[\sum_i(f_i\otimes \eta_i)- \xi \in i^{-1}{\cal I}_Y\otimes \frak{g}, \]
then we have 
\begin{align} 
\label{Hactcompati}
\sum_i (f_i|_Y \otimes 1)((1\otimes \eta_i) v) = (1\otimes \xi) v 
\end{align}
for $v\in {\cal V}_Y$.
In the proposition below,
 we will see that these actions give a well-defined
 $\ovl{\cal A}$-module structure.

Let ${\cal V}:=
{\cal H}om_{\ovl{\cal A}}
({\cal O}_Y \otimes_{i^{-1}{\cal O}_X} i^{-1}U(\wtl{\frak{g}}_X),\,
 {\cal V}_Y),$
namely, ${\cal V}$ 
 consists of the sections
 $v\in {\cal H}om_{\bb{C}}({\cal O}_Y \otimes_{i^{-1}{\cal O}_X}
 i^{-1}U(\wtl{\frak{g}}_X), {\cal V}_Y)$
 satisfying
\begin{align*}
&v((1\otimes \xi)(f\otimes D)) = (1\otimes \xi)(v(f\otimes D)), \\
&v((1\otimes \eta)(f\otimes D)) = (1\otimes \eta)(v(f\otimes D)),
 {\rm \ and\ }\\
&v(f'f\otimes D) = (f'\otimes 1) (v(f\otimes D))
\end{align*}
for $f,f'\in {\cal O}_Y$, $D\in U(\frak{g})$,
 $\eta\in \frak{k}$, and $\xi\in i^{-1}{\cal H}$.
We endow ${\cal V}$ with an $i^{-1}\wtl{\frak{g}}_X$-module structure by
giving $(f\otimes D)\cdot v$ as 
\begin{align*}
((f\otimes D)\cdot v)(f'\otimes D')=v(f' \otimes (1\otimes D')(f\otimes D))
\end{align*}
for $v\in {\cal V}$, 
 $f\in i^{-1}{\cal O}_X$,
 $f'\in{\cal O}_Y$, and $D, D'\in U(\frak{g})$.

\begin{prop}
\label{propconst}
Let $V$ be a $(\frak{h},M)$-module.
Then the left $\ovl{\cal A}$-action on ${\cal V}_Y$
 given above
 is well-defined,
 and the $i^{-1}\wtl{\frak{g}}_X$-module 
\[{\cal V}:=
{\cal H}om_{\ovl{\cal A}}
({\cal O}_Y \otimes_{i^{-1}{\cal O}_X} U(\wtl{\frak{g}}_X),\,
 {\cal V}_Y)\]
is associated with $V$ in the sense of Definition \ref{de:assmod}.
\end{prop}

\begin{proof}
Let $k_0\in K$ and $y_0:=k_0M \in Y$.
We fix a trivialization near $y_0$ in the following way.
Take sections $\xi_1,\dots,\xi_n \in i^{-1} {\cal H}$
 on a neighborhood $U$ of $y_0$ in $Y$ such that the map 
\begin{align*}
(i^{-1}{\cal O}_X)^{\oplus n}|_U \to
 (i^{-1}{\cal H})|_U,\quad 
(f_1,\dots,f_n)\mapsto \sum_{i=1}^n f_i \xi_i
\end{align*}
is an isomorphism.
Take elements $\eta_1,\dots,\eta_m\in \frak{k}$
such that they form a basis of the quotient space
 $\frak{k}/\Ad(k_0)(\frak{m})$
and take $\zeta_1,\dots,\zeta_l \in\frak{g}$ such that
$\eta_1,\dots,\eta_m,\zeta_1,\dots,\zeta_l$
 form a basis of the quotient space
 $\frak{g}/\Ad(i(k_0))\frak{h}$.
Replacing $U$ if necessary, we get an isomorphism
\begin{align}
\label{eqn:vbisom}
(i^{-1}{\cal O}_X)^{\oplus n+m+l}|_U &\to
 (i^{-1}{\cal O}_X\otimes_\bb{C} \frak{g})|_U,\\ \nonumber
(f_1,\dots,f_n,g_1,\dots,g_m,h_1,\dots,h_l)
&\mapsto \sum_{i=1}^n f_i \xi_i 
+\sum_{i=1}^m (g_i\otimes \eta_i)+\sum_{i=1}^l (h_i\otimes \zeta_i).
\end{align}
For integers $s, t\geq 0$, let
\[I_{s,t}:=\{\boldsymbol{i}=(i(1), \dots, i(s)):
 1\leq i(1) \leq \dots \leq i(s) \leq t \},
 \quad I_t := \coprod_{s=0}^\infty I_{s,t}. \]
If $s=0$, the set $I_{0,t}$ consists of one element $()$.
For ${\boldsymbol{i}}=(i(1), \dots, i(s)) \in I_{s,l}$, 
 we put $\zeta_{\boldsymbol{i}}:=1\otimes\zeta_{i(1)}\cdots\zeta_{i(s)}
 \in i^{-1}{\cal O}_X\otimes U(\frak{g})$.
If $s=0$ and ${\boldsymbol{i}}=()$
 then put $\zeta_{\boldsymbol{i}}:=1\otimes 1$.
In the same way,
 for ${\boldsymbol{i'}}=(i'(1),\dots,i'(s)) \in I_{s,n}$
 and ${\boldsymbol{i''}}=(i''(1),\dots,i''(s)) \in I_{s,m}$,
 put $\xi_{\boldsymbol{i'}}:=\xi_{i'(1)}\cdots\xi_{i'(s)}$ and
 $\eta_{\boldsymbol{i''}}:=1\otimes \eta_{i''(1)}\cdots\eta_{i''(s)}$. 
From the isomorphism \eqref{eqn:vbisom} and
 the Poincar$\rm{\acute{e}}$--Birkhoff--Witt theorem, 
we see that a section of $i^{-1}U(\wtl{\frak{g}}_X)|_U$
is uniquely written as 
\begin{align*}
\sum_{{\boldsymbol{i}} \in I_{l},\, {\boldsymbol{i}'} \in I_{n},\,
 {\boldsymbol{i''}} \in I_{m}}
 f_{\boldsymbol{i, i', i''}}
 \xi_{\boldsymbol{i'}}\eta_{\boldsymbol{i''}}\zeta_{\boldsymbol{i}},
\end{align*}
where $f_{\boldsymbol{i, i',i''}}
 \in i^{-1}{\cal O}_X$,
 and  $f_{\boldsymbol{i, i', i''}}=0$ except for
 finitely many $({\boldsymbol{i, i', i''}})$.
Hence a section of
 $({\cal O}_Y \otimes_{i^{-1}{\cal O}_X}i^{-1}U(\wtl{\frak{g}}_X))|_U$
is uniquely written as a finite sum
$\sum_{\boldsymbol{i, i', i''}}
 f_{\boldsymbol{i, i',i''}}
 \xi_{\boldsymbol{i'}}\eta_{\boldsymbol{i''}}\zeta_{\boldsymbol{i}}$
for $f_{\boldsymbol{i, i',i''}}\in {\cal O}_Y$.

\begin{lem}
\label{atriv}
The subsheaf $\ovl{\cal A}|_U$ of
 ${\cal O}_Y \otimes_{i^{-1}{\cal O}_X}
 i^{-1}U(\wtl{\frak{g}}_X)$
 consists of the sections written as a finite sum
\begin{align*}
\sum_{{\boldsymbol{i'}}\in I_n,\, {\boldsymbol{i''}}\in I_m}
 f_{\boldsymbol{i',i''}}\otimes \xi_{\boldsymbol{i'}}\eta_{\boldsymbol{i''}}
\end{align*}
for $f_{\boldsymbol{i',i''}}\in {\cal O}_Y$.
\end{lem}

\begin{proof}
It is enough to prove that
 for any section $a\in {\cal A}|_U$
 there exist functions $f_{\boldsymbol{i', i''}}\in i^{-1}{\cal O}_X$
 such that 
\begin{align}
\label{stdex}
a -\sum_{\boldsymbol{i', i''}} f_{\boldsymbol{i', i''}}
 \xi_{\boldsymbol{i'}}\eta_{\boldsymbol{i''}}
 \in i^{-1}{\cal I}_Y \otimes U(\frak{g}). 
\end{align}
For this we observe relations in the algebra $i^{-1}U(\wtl{\frak{g}}_X)$.
By our choice of $\xi_1,\dots,\xi_n$ and $\eta_1,\dots,\eta_m$,
 we can find $f_i, g_i \in i^{-1}{\cal O}_X$ for each $\eta\in\frak{k}$
 such that
\[
(1\otimes \eta)-
\biggl(\sum_{i=1}^n f_i\xi_i + \sum_{i=1}^m g_i\otimes \eta_i \biggr)
 \in i^{-1}{\cal I}_Y \otimes U(\frak{g}). \]
We also have 
\[
[\xi_i, f\otimes 1]=0,
\quad [1\otimes \eta, 1\otimes \eta']=1\otimes[\eta,\eta'], 
\quad [1\otimes \eta, f\otimes 1]=(\eta_X (f))\otimes 1
\]
for $f\in i^{-1}{\cal O}_X$, $\eta,\eta'\in\frak{k}$.
Further $[\xi_i,\xi_j], [1\otimes \eta_i, \xi_j]\in i^{-1}{\cal H}$
and hence there exist
 $f_{i,j,k}, g_{i,j,k}\in i^{-1}{\cal O}_X$ such that 
\begin{align*}
[\xi_i, \xi_j]=\sum_{k=1}^n f_{i,j,k} \xi_k,\quad
[1\otimes \eta_i, \xi_j]=\sum_{k=1}^n g_{i,j,k} \xi_k.
\end{align*}
Since ${\cal A}$ is generated by $i^{-1}{\cal H}$, $1\otimes \frak{k}$
 and $i^{-1}{\cal O}_X\otimes 1$, we can prove \eqref{stdex}
 by using these relations iteratively and
 using 
 ${\cal A}(i^{-1}{\cal I}_Y \otimes U(\frak{g}))
 \subset i^{-1}{\cal I}_Y\otimes U(\frak{g})$.
%\qed
\end{proof}

From the lemma above and its proof,
 we see that the algebra $\ovl{\cal A}$ is generated by
 ${\cal O}_Y\otimes 1$, $1\otimes \xi_1,\dots, 1\otimes\xi_n$,
 and $1\otimes \frak{k}$ with the relations:
\begin{align*}
&1\otimes \eta=
\sum_{i=1}^n f_i\otimes \xi_i + \sum_{i=1}^m g_i\otimes \eta_i,\\
&[1\otimes \xi_i, f\otimes 1]=0,
\quad [1\otimes \eta, 1\otimes \eta']=1\otimes[\eta,\eta'], 
\quad [1\otimes \eta, f\otimes 1]=(\eta_Y (f))\otimes 1,\\
&[1\otimes \xi_i, 1\otimes \xi_j]=\sum_{k=1}^n f_{i,j,k} \otimes \xi_k,\quad
[1\otimes \eta_i, 1\otimes \xi_j]=\sum_{k=1}^n g_{i,j,k} \otimes \xi_k,
\end{align*}
where $f_i, g_i, f_{i,j,k}, g_{i,j,k}$ are the restrictions to $Y$
 of the corresponding functions in the proof of Lemma~\ref{atriv}
 and $f\in {\cal O}_Y$, $\eta, \eta'\in\frak{k}$.
We can check that these relations are compatible
 with the action on ${\cal V}_Y$ (see \eqref{Hactcompati})
 and hence the $\ovl{\cal A}$-action
 on ${\cal V}_Y$ is well-defined.

By Lemma~\ref{atriv},
$({\cal O}_Y\otimes_{i^{-1}{\cal O}_X}
 i^{-1}U(\wtl{\frak{g}}_X))|_U$
is a free $\ovl{\cal A}|_U$-algebra
with basis $1\otimes \zeta_{\boldsymbol{i}}$.
Therefore, the map 
\begin{align*}
\phi:{\cal V}|_U \to
\prod_{{\boldsymbol{i}} \in I_l}
{\cal V}_Y|_U
\end{align*}
given by
$\phi(v)=(v(1\otimes \zeta_{\boldsymbol{i}}))_{\boldsymbol{i}}$
is bijective.

Our choice of $\zeta_1,\dots,\zeta_l$ implies that
 they form a basis of the normal tangent bundle of $U$ in $X$.
Since $\phi$ is bijective, we see that 
\begin{align*}
\phi((i^{-1}{\cal I}_Y)^p {\cal V}|_U) =
\prod_{s=p}^\infty\  \prod_{{\boldsymbol{i}} \in I_{s,l}}
{\cal V}_Y|_U,
\end{align*}
and hence
\begin{align*}
({\cal V}/(i^{-1}{\cal I}_Y)^p {\cal V})|_U \simeq
\prod_{s=0}^{p-1}\ \prod_{{\boldsymbol{i}} \in I_{s,l}}
{\cal V}_Y|_U.
\end{align*}
If we endow the right side of the last isomorphism with
 ${\cal O}_{Y_p}$-module structure via the isomorphism, 
it is written as follows.
Let $f\in i^{-1}{\cal O}_X$ and
 $v=(v_{\boldsymbol{i}})_{\boldsymbol{i}}$.
For a subset $A \subset \{1,\dots, s\}$ with 
 $A=\{a(1),\dots, a(t)\}$, $a(1) < \dots < a(t)$ and
 for ${\boldsymbol{i}}=(i(1),\dots,i(s))\in I_{s,l}$, 
 let $\{b(1),\dots, b(s-t)\}=\{1,\dots, s\}\setminus A$ with
 $b(1) < \dots < b(s-t)$ and
 put ${\boldsymbol{i'}}:=(i(b(1)),\dots, i(b(s-t))) \in I_{s-t,l}$.
Then the ${\boldsymbol{i}}$-term of $f\cdot v$ is
 given as
\begin{align}
\label{eqn:modstr}
(f\cdot v)_{\boldsymbol{i}}=
\sum_{A\subset \{1,\dots, s \}}
((\zeta_{i(a(1))})_X\cdots(\zeta_{i(a(t))})_X f)|_U \cdot 
v_{{\boldsymbol{i}}'}.
\end{align}
On the right side here, we use the ${\cal O}_Y$-action on ${\cal V}_Y$.
This $i^{-1}{\cal O}_X$-action on
 $\prod_{s=0}^{p-1}\ \prod_{{\boldsymbol{i}} \in I_{s,l}}{\cal V}_Y|_U$
induces an ${\cal O}_{Y_p}$-action.

We now show that ${\cal V}/(i^{-1}{\cal I}_Y)^p{\cal V}$
 is a quasi-coherent and flat ${\cal O}_{Y_p}$-module.
Suppose first that ${\cal V}_Y|_U$ is a free ${\cal O}_U$-module on $U$ 
 so there exist sections $v_j\in \Gamma(U,{\cal V}_Y)$, $j\in J$
 such that the map
 ${\cal O}_U^{\oplus J} \to {\cal V}_Y|_U$,
 $(f_j)_{j\in J} \mapsto \sum_{j\in J} f_j v_j$
 is bijective.
We define the map
\begin{align*}
\psi:
({\cal O}_{Y_p}|_U)^{\oplus J}
\to \prod_{s=0}^{p-1}\ \prod_{{\boldsymbol{i}}\in I_{s,l}}
{\cal V}_Y|_U
\end{align*}
by giving the ${\boldsymbol{i}}$-term of $\psi(f)$
 for ${\boldsymbol{i}}=(i(1),\dots,i(s))\in I_{s,l}$ and $f=(f_j)_{j\in J}$ as
\begin{align*}
\psi(f)_{\boldsymbol{i}}
=\sum_{j\in J}
((\zeta_{i (1)})_X\cdots(\zeta_{i (s)})_X  f_j)|_U \cdot v_j.
\end{align*}
Then $\psi$ is an isomorphism of
 ${\cal O}_{Y_p}|_U$-modules and hence
 $({\cal V}/(i^{-1}{\cal I}_Y)^p{\cal V})|_U$ is a free
 ${\cal O}_{Y_p}|_U$-module.

For general case, we write $V$ as a union of finite-dimensional
 $M$-submodules: $V=\bigcup_\alpha V^\alpha$.
Then the $K$-equivariant quasi-coherent ${\cal O}_Y$-module ${\cal V}^\alpha_Y$
  with fiber $V^\alpha$ is locally free.
If we define the ${\cal O}_{Y_p}$-module structure on
$\prod_{s=0}^{p-1}\ \prod_{{\boldsymbol{i}}\in I_{s,l}}
{\cal V}^\alpha_Y|_U$
 as in \eqref{eqn:modstr},
then the preceding argument proves that
 it is a locally free ${\cal O}_{Y_p}|_U$-module.
Since ${\cal V}_Y$ is the union of ${\cal V}^\alpha_Y$, we see that
 $({\cal V}/(i^{-1}{\cal I}_Y)^p{\cal V})|_U$ is isomorphic to the union of 
$\prod_{s=0}^{p-1}\ \prod_{{\boldsymbol{i}}\in I_{s,l}}
{\cal V}^\alpha_Y|_U$ as an ${\cal O}_{Y_p}|_U$-module. 
Hence
 ${\cal V}/(i^{-1}{\cal I}_Y)^p{\cal V}$ is
 a quasi-coherent and flat 
 ${\cal O}_{Y_p}$-module.

We define a $K$-action on ${\cal V}$ by 
\[(k\cdot v)(f\otimes D)=
k\cdot(v((k^{-1}\cdot f)\otimes \Ad(i(k)^{-1})D))
\]
for $k\in K$, $v\in {\cal V}$, $f\in {\cal O}_Y$,
 and $D\in U(\frak{g})$.
This action descends to a $K$-action on
 ${\cal V}/(i^{-1}{\cal I}_Y)^p{\cal V}$ and 
 makes it
 a $K$-equivariant ${\cal O}_{Y_p}$-module.
From this definition,
 it immediately follows that
 the maps
${\cal V}/(i^{-1}{\cal I}_Y)^p{\cal V}\to 
 {\cal V}/(i^{-1}{\cal I}_Y)^{p-1}{\cal V}$
and 
$i^{-1}\wtl{\frak{g}}_X \otimes {\cal V}/(i^{-1}{\cal I}_Y)^p{\cal V}
 \to{\cal V}/(i^{-1}{\cal I}_Y)^{p-1}{\cal V}$
commute with $K$-actions for all $p>0$.

We have checked conditions (1), (2) and (3) of Definition \ref{de:assmod}.
We can verify the condition (4) by computing the $\frak{k}$-action as 
\begin{align*}
(\eta \cdot v)(f\otimes D)&=v(f\otimes D\eta)\\
&=-v(f \otimes [\eta,D])+v((1 \otimes \eta)(f\otimes D))
-v((\eta_Y (f))\otimes D)\\
&=-v(f \otimes [\eta, D])+(1 \otimes \eta)(v(f\otimes D))
-v((\eta_Y (f)) \otimes D)
\end{align*}
for $\eta\in \frak{k}$, $v\in {\cal V}$, $f\in {\cal O}_Y$, and
 $D\in U(\frak{g})$.

For the condition (5), we get an isomorphism of vector spaces
 $\iota:{\cal V}/(i^{-1}{\cal I}_o){\cal V}\simeq V$
 by taking fiber of the isomorphism
 $\phi:{\cal V}/(i^{-1}{\cal I}_Y){\cal V}\simeq {\cal V}_Y$
 at $o$.
The map $\iota$ is written as
 $\iota(v)=(v(1\otimes 1))(e)$ for $v\in {\cal V}$.
For $\xi\in\frak{h}$, there exists a section
 $\xi' \in i^{-1}{\cal H}$ near the base point $o$ such that
 $1\otimes \xi - \xi' \in i^{-1}{\cal I}_o\otimes \frak{g}$, or equivalently,
 $\xi'(e)=\xi$.
Then 
\begin{align*}
\iota(\xi v)=((\xi v)(1\otimes 1))(e)
=(v(1\otimes \xi))(e)
=(v(\xi'))(e)=\xi (v(1\otimes 1)(e))
=\xi \iota(v).
\end{align*}
Moreover, we have
\begin{align*}
\iota(m v)=((m v)(1\otimes 1))(e)
=(m(v(1\otimes 1)))(e)
=m (v(1\otimes 1)(e))
=m \iota (v)
\end{align*}
for $m\in M$
and hence $\iota$ commutes with $(\frak{h},M)$-actions.
%\qed
\end{proof}

\begin{rem}
The $i^{-1}\wtl{\frak{g}}_X$-module ${\cal V}$
 constructed above in this section
 has the following universal property.
If ${\cal V}'$ is another $i^{-1}\wtl{\frak{g}}_X$-module
 associated with $V$, then there exists a canonical map
 ${\cal V}'\to {\cal V}$
 such that the induced map
\[V\simeq {\cal V}'/(i^{-1}{\cal I}_o){\cal V}'
 \to {\cal V}/(i^{-1}{\cal I}_o){\cal V}\simeq V\]
 is the identity map.
Moreover, it also induces an isomorphism
\[{\cal V}'/(i^{-1}{\cal I}_Y)^p{\cal V}'
\to{\cal V}/(i^{-1}{\cal I}_Y)^p{\cal V}
\]
for any $p\in\bb{N}$.
Therefore, the tensor product
 $i^{-1}i_+{\cal L}\otimes_{i^{-1}{\cal O}_X} {\cal V}'$
 does not depend on the choice of ${\cal V}'$
 up to canonical isomorphism.
We will give another description of the $i^{-1}\wtl{\frak{g}}_X$-module
 $i^{-1}i_+{\cal L}\otimes_{i^{-1}{\cal O}_X} {\cal V}$
 in Proposition~\ref{prop:isom}.
\end{rem}

\section{Twisted ${\cal D}$-modules}
\label{sec:tdo}

Retain the notation of the previous sections.
Let $V$ be a $(\frak{h},M)$-module and ${\cal V}$
 an $i^{-1}\wtl{\frak{g}}_X$-module associated with $V$.
Since ${\cal V}/(i^{-1}{\cal I}_Y){\cal V}$
 is a $K$-equivariant quasi-coherent ${\cal O}_Y$-module
 with typical fiber $V$, there is a canonical isomorphism
 ${\cal V}/(i^{-1}{\cal I}_Y){\cal V}\simeq {\cal V}_Y$.
We view ${\cal H}:= \ker\, (p:{\cal O}_X \otimes \frak{g} \to {\cal T}_X)$
 as a subsheaf of $U(\wtl{\frak{g}}_X)$.
Since ${\cal H}({\cal I}_Y\otimes U(\frak{g})) \subset
 {\cal I}_Y\otimes U(\frak{g})$, the $i^{-1}{\cal H}$-action on 
 ${\cal V}$ induces one on ${\cal V}/ (i^{-1}{\cal I}_Y) {\cal V}$.
By regarding local sections of these equivariant modules as
 vector-valued regular functions, this action is written as
\begin{align}
\label{Haction}
(\xi v)(k)=\xi(i(k))v(k)
\end{align}
for $\xi \in i^{-1}{\cal H}$, $v\in {\cal V}$ and $k\in K$.
Indeed, since the action map
 $i^{-1}{\cal H}\otimes {\cal V}/(i^{-1}{\cal I}_Y){\cal V}
 \to {\cal V}/(i^{-1}{\cal I}_Y){\cal V}$
 commutes with $K$-actions by Definition~\ref{de:assmod} (3),
 it is enough to prove \eqref{Haction} for $k=e$.
This follows from
 ${\cal H}({\cal I}_o \otimes U(\frak{g}))
\subset {\cal I}_o\otimes U(\frak{g})$
 and Definition~\ref{de:assmod} (5).

The ${\cal O}_Y$-modules ${\cal L}$, ${\cal V}_Y$, $\Omega_Y$,
 and $i^*\Omega_X\spcheck$ are $K$-equivariant with typical fiber
 $\bigwedge^{\rm top}(\frak{k}/\frak{l})$, $V$,
 $\bigwedge^{\rm top}(\frak{k}/\frak{m})^*$,
 and $\bigwedge^{\rm top}(\frak{g}/\frak{h})$, respectively.
Hence the tensor product
 ${\cal L}\otimes_{{\cal O}_Y} {\cal V}_Y
 \otimes_{{\cal O}_Y} \Omega_Y \otimes_{{\cal O}_Y} i^*\Omega_X\spcheck$
 is also $K$-equivariant and has typical fiber
 $\bigwedge^{\rm top}(\frak{k}/\frak{l})\otimes 
 V\otimes \bigwedge^{\rm top}(\frak{k}/\frak{m})^*\otimes 
\bigwedge^{\rm top}(\frak{g}/\frak{h})$.
We give a right
 $i^{-1}{\cal H}$-module structure,
 a right $\frak{k}$-module structure,
 and a right ${\cal O}_Y$-module structure on
 the sheaf ${\cal L} \otimes_{{\cal O}_Y} 
 {\cal V}_Y\otimes_{{\cal O}_Y} \Omega_Y \otimes_{{\cal O}_Y}
 i^{*}\Omega_X\spcheck$
 by
\begin{align*}
&((f\otimes v\otimes \omega \otimes \omega')\xi)(k)
 = - f(k) \otimes (\xi(i(k))v(k)) \otimes \omega(k) \otimes \omega'(k)\\
&\qquad\qquad\qquad\qquad\qquad
  - f(k) \otimes v(k) \otimes \omega(k) \otimes \ad(\xi(i(k)))\omega'(k), \\
&(f\otimes v\otimes \omega \otimes \omega')\eta 
 = -(\eta f)\otimes v\otimes \omega \otimes \omega' 
 -f\otimes(\eta v)\otimes \omega \otimes \omega' \\
&\qquad\qquad\qquad\qquad\qquad
 - f\otimes v\otimes (\eta\omega) \otimes \omega' 
 - f\otimes v\otimes \omega \otimes (\eta\omega'), \\
&(f\otimes v\otimes \omega \otimes \omega')f'
 =f'f\otimes v\otimes \omega \otimes \omega'
\end{align*}
for $f\in {\cal L}$, $\xi\in i^{-1}{\cal H}$, $\eta\in\frak{k}$,
 $v\in {\cal V}_Y$, $\omega\in \Omega_Y$, 
 $\omega'\in i^{*}\Omega_X\spcheck$,
 $f'\in {\cal O}_Y$, and $k\in K$.
These actions are compatible:
if $f_i\in i^{-1}{\cal O}_X$, $\eta_i\in\frak{k}$
 and $\xi \in i^{-1}{\cal H}$ satisfy 
\[\sum_i(f_i \otimes \eta_i)- \xi \in
 i^{-1}{\cal I}_Y\otimes U(\frak{g}),\]
then we have 
\begin{align*} 
\sum_i\bigl((f\otimes v\otimes \omega \otimes \omega')f_i|_Y\bigr) \eta_i
 = (f\otimes v\otimes \omega \otimes \omega')\xi.  
\end{align*}
Therefore, we can prove in the same way
 as in Section~\ref{sec:const} that
 these actions define a right $\ovl{\cal A}$-module
 structure on 
 ${\cal L} \otimes_{{\cal O}_Y} 
 {\cal V}_Y\otimes_{{\cal O}_Y} \Omega_Y \otimes_{{\cal O}_Y}
 i^{*}\Omega_X\spcheck$.

By using this right $\ovl{\cal A}$-module structure,
 we consider the sheaf
\[({\cal L} \otimes_{{\cal O}_Y} 
 {\cal V}_Y\otimes_{{\cal O}_Y} \Omega_Y \otimes_{{\cal O}_Y}
 i^{*}\Omega_X\spcheck)
\otimes_{\ovl{\cal A}} 
({\cal O}_Y \otimes_{i^{-1}{\cal O}_X} i^{-1}U(\wtl{\frak{g}}_X)),\]
which has
 a right $i^{-1}\wtl{\frak{g}}_X$-module structure.
We view it as a left $i^{-1}\wtl{\frak{g}}_X$-module
 via the anti-isomorphism
\begin{align*}
S:U(\wtl{\frak{g}}_X)\to U(\wtl{\frak{g}}_X),\quad 
f\otimes 1 \mapsto f\otimes 1,\quad 1\otimes \xi \mapsto -1\otimes \xi
\end{align*}
for $f\in {\cal O}_X$, $\xi\in\frak{g}$.

\begin{prop}
\label{prop:isom}
Let ${\cal L}$ be as in Section \ref{sec:loc}.
Let ${\cal V}$ be an $i^{-1}\wtl{\frak{g}}_X$-module
 associated with a $(\frak{h},M)$-module $V$.
Then there exists a $K$-equivariant
 isomorphism of $i^{-1}\wtl{\frak{g}}_X$-modules 
\begin{align*}
i^{-1}i_+{\cal L}\otimes_{i^{-1}{\cal O}_X} {\cal V}\simeq 
({\cal L} \otimes_{{\cal O}_Y} 
 {\cal V}_Y\otimes_{{\cal O}_Y} \Omega_Y \otimes_{{\cal O}_Y}
 i^{*}\Omega_X\spcheck)
\otimes_{\ovl{\cal A}} 
({\cal O}_Y \otimes_{i^{-1}{\cal O}_X} i^{-1}U(\wtl{\frak{g}}_X)).
\end{align*}
\end{prop}

\begin{proof}
Let $F_p i^{-1}i_+{\cal L}$ be the filtration of $i^{-1}i_+{\cal L}$ 
 as in Section \ref{sec:loc}.
Then 
 $F_0i^{-1}i_+{\cal L}\otimes_{i^{-1}{\cal O}_X}{\cal V}$
 is regarded as a subsheaf of 
 $i^{-1}i_+{\cal L}\otimes_{i^{-1}{\cal O}_X}{\cal V}$
 (see Remark~\ref{rem:gkact}).
We have 
\begin{align*}
F_0i^{-1}i_+{\cal L}\otimes_{i^{-1}{\cal O}_X}{\cal V}
&\simeq 
F_0i^{-1}i_+{\cal L}\otimes_{{\cal O}_Y}{\cal V}/(i^{-1}{\cal I}_Y) {\cal V}\\
&\simeq
{\cal L}\otimes_{{\cal O}_Y} \Omega_Y \otimes_{{\cal O}_Y} i^*\Omega_X\spcheck
\otimes_{{\cal O}_Y} {\cal V}/(i^{-1}{\cal I}_Y){\cal V}.
\end{align*}
Therefore, we get an isomorphism of $K$-equivariant ${\cal O}_Y$-modules
\begin{align}
\label{F0isom}
\psi_0:\ 
&{\cal L}\otimes_{{\cal O}_Y} {\cal V}_Y  \otimes_{{\cal O}_Y}
\Omega_Y \otimes_{{\cal O}_Y} i^*\Omega_X\spcheck
\xrightarrow{\sim} 
F_0i^{-1}i_+{\cal L}\otimes_{i^{-1}{\cal O}_X}{\cal V},
 \\ \nonumber
&f\otimes v\otimes \omega \otimes \omega'
\mapsto (f\otimes \omega \otimes \omega')\otimes v.
\end{align}
Here $v\in {\cal V}_Y$ and we choose a section of ${\cal V}$
 that is sent to 
 $v\in {\cal V}_Y\simeq {\cal V}/(i^{-1}{\cal I}_Y){\cal V}$
 by the quotient map, which we denote by the same letter $v\in {\cal V}$.
Write ${\cal V}'_Y:=
{\cal L}\otimes_{{\cal O}_Y} {\cal V}_Y  \otimes_{{\cal O}_Y}
\Omega_Y \otimes_{{\cal O}_Y} i^*\Omega_X\spcheck$
for simplicity.
The isomorphism \eqref{F0isom}
 extends to the homomorphism of
 $i^{-1}\wtl{\frak{g}}_X$-modules
\begin{align*}
\psi:\ 
&{\cal V}'_Y
\otimes_\bb{C}
 ({\cal O}_Y \otimes_{i^{-1}{\cal O}_X} i^{-1}U(\wtl{\frak{g}}_X))
\to 
i^{-1}i_+{\cal L}\otimes_{i^{-1}{\cal O}_X}{\cal V},\\ \nonumber
&v \otimes (1\otimes (f\otimes D)) \mapsto S(f\otimes D)\cdot \psi_0(v). 
\end{align*}
We can check that
 the map $\psi$ descends to 
\begin{align*}
\ovl{\psi}:
{\cal V}'_Y \otimes_{\ovl{\cal A}}
 ({\cal O}_Y \otimes_{i^{-1}{\cal O}_X}i^{-1}U(\wtl{\frak{g}}_X))
\to 
i^{-1}i_+{\cal L}\otimes_{i^{-1}{\cal O}_X}{\cal V}.
\end{align*}
Let
\begin{align*}
\pi:
{\cal V}'_Y \otimes_\bb{C}
 ({\cal O}_Y \otimes_{i^{-1}{\cal O}_X}i^{-1}U(\wtl{\frak{g}}_X))
\to
{\cal V}'_Y \otimes_{\ovl{\cal A}}
 ({\cal O}_Y \otimes_{i^{-1}{\cal O}_X}i^{-1}U(\wtl{\frak{g}}_X))
\end{align*}
be the quotient map and put 
\begin{align*}
{\cal V}_p:=
\pi\left( 
{\cal V}'_Y \otimes_\bb{C} ({\cal O}_Y\otimes_\bb{C} U_p(\frak{g}))
\right),
\end{align*}
 where $\{U_p(\frak{g})\}_{p\in \bb{N}}$
 is the standard filtration of $U(\frak{g})$.
We have 
\begin{align*}
\ovl{\psi}({\cal V}_p)
=\psi({\cal V}'_Y \otimes_\bb{C} ({\cal O}_Y\otimes_\bb{C} U_p(\frak{g}))
\subset F_pi^{-1}i_+{\cal L}\otimes_{i^{-1}{\cal O}_X}{\cal V}.
\end{align*}
Let us take an open set $U\subset Y$ and elements
 $\zeta_1,\dots,\zeta_l\in\frak{g}$ as in the proof of
 Proposition~\ref{propconst} and use the same notation.
Then by an argument similar to the proof of Proposition~\ref{propconst},
 we obtain a bijective map of sheaves
\begin{align*}
&\prod_{s=0}^p \prod_{{\boldsymbol{i}}\in I_{s,l}}
 {\cal V}_Y'|_U \simeq {\cal V}_p|_U,\\
& (v_{\boldsymbol{i}})_{\boldsymbol{i}} \mapsto \sum_{\boldsymbol{i}}
 \pi(v_{\boldsymbol{i}} \otimes (1\otimes \zeta_{\boldsymbol{i}}))
\end{align*}
and hence we have 
\begin{align*}
\prod_{{\boldsymbol{i}}\in I_{p,l}}
 {\cal V}_Y'|_U \simeq {\cal V}_p/{\cal V}_{p-1}|_U.
\end{align*}
We also see that
\begin{align*}
(F_p i^{-1}i_+{\cal L}\otimes_{i^{-1}{\cal O}_X} {\cal V})/
(F_{p-1} i^{-1}i_+{\cal L}\otimes_{i^{-1}{\cal O}_X} {\cal V})
\simeq
 (F_p i^{-1}i_+{\cal L}/F_{p-1} i^{-1}i_+{\cal L})
 \otimes_{{\cal O}_Y} {\cal V}_Y
\end{align*}
and 
\begin{align*}
F_p i^{-1}i_+{\cal L}/F_{p-1} i^{-1}i_+{\cal L}
\simeq {\cal L} \otimes_{{\cal O}_Y}
\Omega_Y \otimes_{{\cal O}_Y} i^*\Omega_X\spcheck
\otimes_{{\cal O}_Y} i^{-1}(({\cal I}_Y)^p/({\cal I}_Y)^{p+1})
\end{align*}
by \cite[Lemma 3.3]{Os11}. 
Since $\zeta_{\boldsymbol{i}}$ for ${\boldsymbol{i}} \in I_{p,l}$
 give a trivialization
 of $i^{-1}(({\cal I}_Y)^p/({\cal I}_Y)^{p+1})$, we conclude that
 the map
\begin{align*}
{\cal V}_p/{\cal V}_{p-1} \to 
(F_p i^{-1}i_+{\cal L}\otimes {\cal V})/
(F_{p-1} i^{-1}i_+{\cal L}\otimes {\cal V})
\end{align*}
 induced by $\ovl{\psi}$ on the successive quotient is an isomorphism.
Therefore the map $\ovl{\psi}$ is also an isomorphism.
We can also see that $\ovl{\psi}$ commutes with $K$-action.
Hence the proposition follows.
%\qed
\end{proof}

Let $\lambda\in\frak{h}^*$ such that
 $\Ad^*(h)\lambda=\lambda$ for $h\in H$. 
For a section $\xi\in {\cal H}$, 
 we define a function $f_{\xi, \lambda}\in {\cal O}_X$ as 
\[
f_{\xi, \lambda}(gH)=\lambda(\xi(g)).
\]
Let ${\cal I}_\lambda$ be the two-sided ideal of
 the sheaf $U(\wtl{\frak{g}}_X)={\cal O}_X\otimes U(\frak{g})$ generated by
$\xi-(f_{\xi, \lambda}\otimes 1)$ for all $\xi\in{\cal H}$.
We define the ring of twisted differential operators as
\[
{\cal D}_{X,\lambda}
:=U(\wtl{\frak{g}}_X)/{\cal I}_\lambda.
\]
Let $\mu:=\lambda|_\frak{m}$
 and define ${\cal D}_{Y,\mu}$ similarly.
Then we can define the direct image of
 a left ${\cal D}_{Y,\mu}$-module ${\cal M}$ by
\begin{align*}
i_+{\cal M}:= i_*\bigl(({\cal M}\otimes_{{\cal O}_Y} \Omega_Y)
 \otimes_{{\cal D}_{Y,-\mu}}
 i^*{\cal D}_{X, -\lambda}\bigr) \otimes_{{\cal O}_X} \Omega_X\spcheck.
\end{align*}

Suppose that $V$ is a $(\frak{h},M)$-module
 and $\frak{h}$ acts on $V$ by $\lambda\in\frak{h}^*$.
The $K$-equivariant ${\cal O}_Y$-module
 ${\cal L} \otimes_{{\cal O}_Y} {\cal V}_Y$
 has a natural structure of left ${\cal D}_{Y,\mu}$-module.
Therefore, we can define the direct image
 $i_+({\cal L}\otimes_{{\cal O}_Y} {\cal V}_Y)$
 as a left ${\cal D}_{X,\lambda}$-module.

\begin{prop}
\label{prop:tdo}
Suppose that $V$ is a $(\frak{h},M)$-module and $\frak{h}$
 acts on $V$ by $\lambda\in\frak{h}^*$
 such that $\Ad^*(h)\lambda=\lambda$ for $h\in H$.
Let ${\cal V}$ be an $i^{-1}\wtl{\frak{g}}_X$-module associated with $V$.
Then we have a $K$-equivariant isomorphism of $i^{-1}\wtl{\frak{g}}_X$-modules
\begin{align*}
i^{-1}i_+{\cal L}\otimes_{i^{-1}{\cal O}_X} {\cal V}
\simeq i^{-1}i_+({\cal L}\otimes_{{\cal O}_Y} {\cal V}_Y).
\end{align*}
\end{prop}

\begin{proof}
We define a filtration
$F_pi^{-1}i_+({\cal L}\otimes_{{\cal O}_Y} {\cal V}_Y)$ of
 $i^{-1}i_+({\cal L}\otimes_{{\cal O}_Y} {\cal V}_Y)$ in the same way
 as $F_pi^{-1}i_+{\cal L}$.
Then
\[
F_0i^{-1}i_+({\cal L}\otimes_{{\cal O}_Y} {\cal V}_Y)\simeq 
{\cal L}\otimes_{{\cal O}_Y} {\cal V}_Y
 \otimes_{{\cal O}_Y} \Omega_Y \otimes_{{\cal O}_Y} i^*\Omega_X\spcheck.
\]
By using the same argument as in Proposition~\ref{prop:isom},
we define a map of $i^{-1}\wtl{\frak{g}}_X$-modules
\[
{\cal V}'_Y\otimes_{\bb{C}}
 ({\cal O}_Y\otimes_{i^{-1}{\cal O}_X}i^{-1}U(\wtl{\frak{g}}_X))
\to i^{-1}i_+ ({\cal L}\otimes_{{\cal O}_Y}{\cal V}_Y)
\]
and we see that it induces an isomorphism
\[
{\cal V}'_Y\otimes_{\ovl{\cal A}}
 ({\cal O}_Y\otimes_{i^{-1}{\cal O}_X}i^{-1}U(\wtl{\frak{g}}_X))
\simeq i^{-1}i_+ ({\cal L}\otimes_{{\cal O}_Y}{\cal V}_Y).
\]
Hence 
\[i^{-1}i_+{\cal L} \otimes_{i^{-1}{\cal O}_X} {\cal V}
\simeq i^{-1}i_+ ({\cal L}\otimes_{{\cal O}_Y}{\cal V}_Y)
\]
by Proposition~\ref{prop:isom}.
%\qed
\end{proof}

Recall that ${\cal L}$ is the $K$-equivariant invertible sheaf
 on $Y=K/M$ with typical fiber
 $\bigwedge^{\rm top}(\frak{k}/\frak{l})$.
We view a one-dimensional vector space
 $\bigwedge^{\rm top}(\frak{k}/\frak{l})^*$
 as a $(\frak{h},M)$-module in the following way:
 $\frak{h}$ acts as zero; the Levi component
 $L$ of $M$ acts as the coadjoint action $\bigwedge {\Ad}^*$;
 the unipotent radical $U$ of $M$ acts trivially.
Let ${\cal L}'$ be an $i^{-1}\wtl{\frak{g}}_X$-module
 associated with $\bigwedge^{\rm top}(\frak{k}/\frak{l})^*$.
Then ${\cal L}'/ (i^{-1}{\cal I}_Y) {\cal L}'$
 is isomorphic to the dual of ${\cal L}$.
Therefore, by Proposition~\ref{prop:tdo} we have
\[i^{-1}i_+{\cal L} \otimes_{i^{-1}{\cal O}_X}
 {\cal V} \otimes_{i^{-1}{\cal O}_X} {\cal L}'
 \simeq i^{-1}i_+{\cal V}_Y.\]
Example~\ref{tensor} shows that
the $i^{-1}\wtl{\frak{g}}_X$-module 
 ${\cal V} \otimes_{i^{-1}{\cal O}_X} {\cal L}'$
is associated with $V\otimes \bigwedge^{\rm top}(\frak{k}/\frak{l})^*$.
\begin{thm}
\label{thm:tdo}
In Setting~\ref{setting}, we assume that $K$ is reductive.
Suppose that $V$ is a $(\frak{h},M)$-module and $\frak{h}$
 acts on $V$  by $\lambda\in\frak{h}^*$
 such that $\Ad^*(h)\lambda=\lambda$ for $h\in H$.
Let $M=L\ltimes U$ be a Levi decomposition.
Then
\begin{align*}
{\rm H}^s (Y, i^{-1}i_+{\cal V}_Y)
 &\simeq (P_{\frak{h}, L}^{\frak{g},K})_{u-s}
 \Bigl(V \otimes \bigwedge^{\rm top}(\frak{k}/\frak{l})^*
 \otimes \bigwedge^{\rm top}(\frak{g}/\frak{h})\Bigr)\\
&\simeq (I_{\frak{g},L}^{\frak{g},K})^{y+s}
 P_{\frak{h}, L}^{\frak{g},L}
 \Bigl(V  \otimes \bigwedge^{\rm top}(\frak{g}/\frak{h})\Bigr)
\end{align*}
 for $s\in \bb{N}$, $u=\dim U$, and $y=\dim Y$.
\end{thm}

\begin{proof}
The first isomorphism follows from Theorem~\ref{loccoh} and
 the argument above.
Since the functor $P_{\frak{h}, L}^{\frak{g},L}$ is exact,
 $(P_{\frak{h}, L}^{\frak{g},K})_{u-s}\simeq
 (P_{\frak{g}, L}^{\frak{g},K})_{u-s}
 \circ P_{\frak{h}, L}^{\frak{g},L}$.
Hence the duality (\cite[Theorem 3.5]{KnVo})
\[(P_{\frak{g}, L}^{\frak{g},K})_{\dim (K/L)-s}\Bigl(\, \cdot \otimes
  \bigwedge^{\rm top}(\frak{k}/\frak{l})^*\Bigr)
 \simeq (I_{\frak{g}, L}^{\frak{g},K})^{s}(\cdot)\]
 and $\dim K/L= \dim U + \dim Y$
 give the second isomorphism.
%\qed
\end{proof}

By Theorem~\ref{thm:tdo} we obtain the convergence of spectral sequence
\begin{align}
\label{spseq}
{\rm H}^s (X, R^t i_+{\cal V}_Y)
 &\Rightarrow (P_{\frak{h}, L}^{\frak{g},K})_{u-s-t}
 \Bigl(V \otimes \bigwedge^{\rm top}(\frak{k}/\frak{l})^*
 \otimes \bigwedge^{\rm top}(\frak{g}/\frak{h})\Bigr)\\ \nonumber
&\simeq (I_{\frak{g},L}^{\frak{g},K})^{y+s+t}
 P_{\frak{h}, L}^{\frak{g},L}
 \Bigl(V  \otimes \bigwedge^{\rm top}(\frak{g}/\frak{h})\Bigr).
\end{align}
Here $R^t i_+$ is the higher direct image functor for a twisted
 left ${\cal D}$-module.

We now see that this spectral sequence implies results
 of \cite{HMSW} and \cite{Kit}.

\begin{ex}
Let $G_\bb{R}$ be a connected real semisimple Lie group with a maximal
 compact subgroup $K_\bb{R}$ and the complexified Lie algebra $\frak{g}$.
Let $K$ be the complexification of $K_\bb{R}$ and $G$
 the inner automorphism group of $\frak{g}$.
There is a canonical homomorphism $i:K\to G$, which has finite kernel.
Suppose that $H$ is a Borel subgroup of $G$. 
Let us apply Setting~\ref{setting}.
Then $X=G/H$ is the full flag variety of $\frak{g}$.
Since $L$ is abelian and $K$ is connected, $L$ acts trivially on
 $\bigwedge^{\rm top}(\frak{k}/\frak{l})^*$.
Moreover in this case it is known that $Y$ is affinely embedded in $X$.
Therefore, $R^ti_+ \simeq 0$ for $t>0$ and 
 the spectral sequence \eqref{spseq} collapses.
We thus get \eqref{eq:HMSWisom} and hence the duality theorem
 (Theorem~\ref{dual}).
\end{ex}

\begin{ex}
Let $G_\bb{R}$ be a connected real semisimple Lie group with a maximal
 compact subgroup $K_\bb{R}$.
We define $K$, $G$, and $i:K\to G$ as in the previous example.
Suppose that $H$ is a parabolic subgroup of $G$
 and apply Setting~\ref{setting}.
Then $X=G/H$ is a partial flag variety of $\frak{g}$.
In this case $Y$ is not necessarily affinely embedded in $X$.
Let $\wtl{X}$ be the full flag variety of $\frak{g}$ and
 let $p:\wtl{X}\to X$ be the natural surjective map.
Then we have an isomorphism
 ${\rm H}^s(\wtl{X}, p^*{\cal M}) \simeq {\rm H}^s(X, {\cal M})$
 for any ${\cal O}_X$-module ${\cal M}$.
Hence \eqref{spseq} becomes
\begin{align*}
{\rm H}^s (\wtl{X}, p^* R^t i_+{\cal V}_Y)
 \Rightarrow (I_{\frak{g},L}^{\frak{g},K})^{y+s+t}
 P_{\frak{h}, L}^{\frak{g},L}
 \Bigl(V  \otimes \bigwedge^{\rm top}(\frak{g}/\frak{h})\Bigr),
\end{align*}
 which is \cite[Theorem 25 (6.6)]{Kit}.
\end{ex}

Let $V$ be any $(\frak{h},M)$-module
 and ${\cal V}$ an $i^{-1}\wtl{\frak{g}}_X$-module
 associated with $V$.
Since
 $i^{-1}i_+{\cal L} \otimes_{i^{-1}{\cal O}_X}
 {\cal L}' 
 \simeq i^{-1}i_+{\cal O}_Y$,
we have 
\[i^{-1}i_+{\cal L} \otimes_{i^{-1}{\cal O}_X}
 {\cal L}' \otimes_{i^{-1}{\cal O}_X} {\cal V}
 \simeq i^{-1}i_+{\cal O}_Y \otimes_{i^{-1}{\cal O}_X} {\cal V}.\]
We can thus rewrite Theorem~\ref{loccoh} as 
\begin{thm}
\label{loccoh2}
In Setting~\ref{setting}, we assume that $K$ is reductive.
Let $M=L\ltimes U$ be a Levi decomposition.
Suppose that $V$ is a $(\frak{h}, M)$-module and 
 that ${\cal V}$ is an $i^{-1}\wtl{\frak{g}}_X$-module
 associated with $V$ (Definition~\ref{de:assmod}).
Then
\begin{align*}
{\rm H}^s (Y, i^{-1}i_+{\cal O}_Y
 \otimes_{i^{-1}{\cal O}_X} {\cal V})
 &\simeq (P_{\frak{h}, L}^{\frak{g},K})_{u-s}
 \Bigl(V \otimes \bigwedge^{\rm top}(\frak{k}/\frak{l})^*
 \otimes \bigwedge^{\rm top}(\frak{g}/\frak{h})\Bigr)\\
 &\simeq (I_{\frak{g},L}^{\frak{g},K})^{y+s}
 P_{\frak{h}, L}^{\frak{g},L}
 \Bigl(V \otimes \bigwedge^{\rm top}(\frak{g}/\frak{h})\Bigr)
\end{align*}
for $s\in \bb{N}$, $u=\dim U$, $y=\dim Y$.
\end{thm}

\subsection*{Acknowledgements}
The author was supported by Grant-in-Aid for JSPS Fellows (10J00710).

\renewcommand{\MR}[1]{}

\def\cprime{$'$}
\providecommand{\bysame}{\leavevmode\hbox to3em{\hrulefill}\thinspace}
\providecommand{\MR}{\relax\ifhmode\unskip\space\fi MR }
\providecommand{\MRhref}[2]{%
  \href{http://www.ams.org/mathscinet-getitem?mr=#1}{#2}
}
\providecommand{\href}[2]{#2}

%\bibliographystyle{amsalpha}

%\bibliography{bib_localization}

\end{document}